\documentclass[pdflatex,sn-mathphys-num]{sn-jnl}% Math and Physical Sciences Numbered Reference Style

\usepackage{lipsum}
\usepackage{amsfonts}
\usepackage{graphicx}
\usepackage{epstopdf}
\usepackage{slashed}     % \slashed{\mathcal{D}}
\usepackage{amsmath}     % 
\usepackage{amssymb,bm}     % \mathbb{}
\usepackage{mathtools}
\usepackage{tabularx,ragged2e,booktabs}
\usepackage{tabularx,ragged2e,booktabs}
\usepackage{hyperref}
\usepackage{url}

\usepackage{amsthm}%
\usepackage{mathrsfs}%
\usepackage[title]{appendix}%
\usepackage{xcolor}%
\usepackage{textcomp}%
\usepackage{manyfoot}%
\usepackage{booktabs}%
\usepackage{algorithm}%
\usepackage{algorithmicx}%
\usepackage{algpseudocode}%
\usepackage{listings}%

\usepackage{tikz}
\usetikzlibrary{decorations.markings,arrows.meta}
\theoremstyle{thmstyleone}%
\newtheorem{theorem}{Theorem}%  meant for continuous numbers
\newtheorem{proposition}[theorem]{Proposition}% 
\newtheorem{assumption}[theorem]{Assumption}
\newtheorem{lemma}[theorem]{Lemma}

\theoremstyle{thmstyletwo}%
\newtheorem{remark}{Remark}%

\theoremstyle{thmstylethree}%
\newtheorem{definition}{Definition}%

\DeclareMathOperator{\diag}{diag}

\begin{document}

\title[Article Title]{Distributed Optimization over Simplicial Complexes via Topological Dirac Synchronization}


\author[1,2]{\fnm{Christian D.} \sur{Rodriguez-Camargo}}\email{christian.rodriguez-camargo.21@ucl.ac.uk}

\author[2]{\fnm{Andrés F. } \sur{Urquijo-Rodriguez}}\email{afurquijor@unal.edu.co}

\author*[3]{\fnm{Eduardo} \sur{Mojica-Nava}}\email{eamojican@unal.edu.co}

\affil[1]{\orgdiv{Department of Physics and Astronomy}, \orgname{University College London}, \orgaddress{\street{Bidborough Street 38-50}, \city{London}, \postcode{WC1E 6BT}, \state{London}, \country{United Kingdom}}}

\affil[2]{\orgdiv{Department of Physics}, \orgname{Universidad Nacional de Colombia}, \orgaddress{\street{Carrera 30 No.45-03}, \city{Bogotá}, \postcode{111321}, \state{Bogotá, D.C.}, \country{Colombia}}}

\affil*[3]{\orgdiv{Department of Electrical and Electronics Engineering}, \orgname{Universidad Nacional de Colombia}, \orgaddress{\street{Carrera 30 No.45-03}, \city{Bogotá}, \postcode{111321}, \state{Bogotá, D.C.}, \country{Colombia}}}


\abstract{We introduce a distributed optimization framework for systems with higher-order interactions represented by simplicial and cell complexes. Specializing global topological Dirac synchronization to gradient flows of strongly convex local objectives, we show that the synchronized equilibrium coincides with the minimizer of the aggregate cost, without Lagrange multipliers, dual variables, or saddle-point dynamics. Orientation-equivariant convex terms and cell-dependent dimensional targets allow the optimal state to be nonzero, as determined by joint per-dimension optimality conditions. Our algorithm couples signals across dimensions through the Dirac operator: nodes interact with edges, edges with nodes and faces, and faces with edges, enabling information exchange beyond graph-Laplacian methods. In the isotropic case, a Master Stability Function analysis yields a necessary and sufficient condition for local exponential stability: the product of the directional coupling strengths must be below the squared ratio of the strong convexity parameter to the spectral radius of the complex. A weighted Lyapunov argument accounting for non-normal Dirac coupling establishes global exponential convergence at the same rate. Finally, we identify a topological admissibility condition ensuring invariance of the synchronization manifold and coherent convergence of cells to the optimum, with no counterpart in classical graph-based distributed optimization.}


%%%%%%%%%%%%%%%

\maketitle

%\thanks{This work was partially supported by the VIII Convocatoria para el Desarrollo y Fortalecimiento de los Grupos de Investigaci\'on en Uniminuto (C119-173), the Convocatoria de investigaci\'on para prototipado de tecnolog\'ias (CPT123-200-5220), and the EPSRC (Grants EP/R513143/1 and EP/T517793/1).}

% REQUIRED
%\begin{MSCcodes}
%93A14, 93D20, 90C25, 05E45
%\end{MSCcodes}
%======================================================================
\section{Introduction}
\label{sec:introduction}
%======================================================================

Distributed optimization over networks has been extensively studied during 
the past two decades, with applications ranging from sensor fusion and 
multi-agent coordination to power systems and large-scale machine 
learning~\cite{nedic2009distributed,nedic2018distributed}. The mathematical 
structure underlying these algorithms is the graph Laplacian, whose null 
space encodes the consensus subspace that distributed agents asymptotically 
reach. Extensions of this paradigm to richer communication structures have 
followed a natural progression: from single-layer graphs to multiplex 
networks coupled through a 
supra-Laplacian~\cite{rodriguezcamargo2023consensus,
RODRIGUEZCAMARGO20231217,11312412}, and more recently, to general 
multilayer networks via a tensorial Laplacian 
formalism~\cite{rodriguezcamargo2025}. In each case, however, agents 
interact through pairwise connections and the underlying coupling 
operator remains a Laplacian-type matrix, however layered or weighted.

A growing body of evidence suggests that many systems of interest in 
engineering and the natural sciences exhibit interactions that are 
intrinsically higher-order, in the sense that they involve groups of 
three or more entities simultaneously and cannot be faithfully reduced 
to pairwise relationships~\cite{Abiad_2026, bianconi2021top}. Chemical reactions 
involve multiple reactants~\cite{doi:10.1021/cr100233r}, synaptic signals and link-signals defined on links connecting brain regions~\cite{10.1162/netn_a_00204}, team coordination requires group-level 
alignment~\cite{Santoro2023}, topological machine learning~\cite{10.1093/pnasnexus/pgaf139}, and physical conservation laws naturally apply around loops 
and higher-dimensional structures~\cite{Bianconi_2023, Battiston2021, Bianconi_2024}. The mathematical objects that capture 
such interactions are simplicial and cell complexes, in which signals 
live not only on nodes but also on edges, triangles, and higher-dimensional 
cells~\cite{millan2025top}. The dimension-specific Hodge Laplacians provide diffusion operators 
for these topological signals~\cite{bianconi2021top}, but a more parsimonious 
description is afforded by the topological Dirac 
operator~\cite{calmon2023dirac}: a first-order differential operator that 
couples signals across adjacent dimensions simultaneously and whose square 
recovers the Hodge Laplacians as its diagonal blocks. Recently, conditions under which the Dirac 
operator induces Global Topological Dirac synchronization 
(GTDS) have been established~\cite{carletti2025gtds}, a state in which all topological signals 
share a common time evolution up to orientation-induced sign changes. 
A distinctive feature of GTDS, with no counterpart in Laplacian-based 
synchronization on graphs, is that the algebraic topology of the complex 
imposes a structural admissibility condition on the synchronized states: 
not every complex can support a GTDS.

In this paper, we connect the GTDS framework to distributed optimization. 
The starting observation is that when the local dynamics on each cell of 
the complex are chosen as the gradient flow of a strongly convex 
objective, while signals on adjacent cells remain coupled through the 
Dirac operator, the synchronized equilibrium of the resulting dynamics 
coincides with the global minimizer of the aggregate cost. This yields 
a distributed optimization over simplicial complexes (DOSC) algorithm whose convergence 
emerges directly from topological synchronization, without Lagrange 
multipliers, dual variables, or saddle-point formulations. In contrast 
to all Laplacian-based methods, the Dirac coupling produces a 
cross-dimensional information flow: nodes exchange information with their 
incident edges, edges with both their incident nodes and the triangles 
they bound, and triangles with their bounding edges.

Contributions: The contributions of this paper are threefold.
(i)~We show that the synchronized equilibrium of the Dirac-coupled 
gradient flow is the global minimizer of the aggregate strongly convex 
cost, yielding a \emph{synchronization-as-optimization} principle that 
dispenses with the dual variables of the Lagrangian approach in our 
previous work~\cite{rodriguezcamargo2023consensus,rodriguezcamargo2025}. By coupling an orientation-equivariant convex core 
with dimensional targets tied to the cell orientations, the framework 
drives the complex to a common optimal state that is in general nonzero 
and is determined jointly by the per-dimension optimality conditions, 
rather than being fixed a priori.
(ii)~We give an explicit isotropic stability criterion 
$\gamma_0\gamma_1 < \nu^2/b_{\max}^2$, derived directly from the spectrum 
of the weighted Dirac operator, where $\gamma_0,\gamma_1$ are the two 
directional Dirac coupling strengths, $\nu$ is the strong convexity 
parameter, and $b_{\max}$ is the spectral radius of the complex 
(Theorem \ref{thm:isotropic}); we then establish global exponential convergence 
\emph{at the same rate} through a weighted Lyapunov argument that 
accounts for the non-normality of the Dirac coupling (Theorem \ref{thm:global}).
(iii)~We characterize the topological admissibility condition 
$D\widehat{U}=0$ under which the synchronization manifold is invariant and 
the cells evolve coherently toward the optimum, a topological feasibility 
notion with no analogue in classical graph-based distributed optimization.

Unlike existing distributed optimization approaches based on Laplacian couplings, DOSC algorithm is the first 
distributed optimization method in which the underlying coupling 
operator is a first-order, dimension-coupling differential operator 
rather than a second-order, dimension-preserving Laplacian. 
Table \ref{tab:comparison} situates the proposed algorithm among representative families of 
distributed optimization methods. The framework 
extends to simplicial and cell complexes of arbitrary maximum dimension 
$K \geq 2$; we restrict the exposition to $K=2$ in order to keep the 
notation compact and the topological admissibility analysis transparent.

\begin{table}[htbp]
\footnotesize
\caption{Comparison of representative distributed optimization methods.
The coupling operator order refers to its order as a differential
operator on the complex; ``cross-dimensional'' indicates whether the
operator couples signals of different topological dimensions.}
\label{tab:comparison}
\centering
\begin{tabularx}{\textwidth}{l X X X X X} 
%\begin{tabular}{lccccc}
\hline
Method & Operator & Order & Cross-dim. & Dual-free \\
\hline
Distributed gradient descent~\cite{nedic2009distributed} & Laplacian & 2 & no & yes \\
Exact diffusion / gradient tracking & Laplacian & 2 & no & yes \\
ADMM & Lagrangian & --- & no & no \\
Primal--dual / saddle-point & Lagrangian & 1 & no & no \\
Laplacian consensus optimization & Laplacian & 2 & no & yes \\
DOSC (this work) & Dirac $\slashed{\mathcal{D}}$ & 1 & yes & yes \\
\hline
\end{tabularx}
\end{table}

The remainder of the paper is organized as follows. 
Section \ref{sec:preliminaries} collects the algebraic and topological 
preliminaries: simplicial complexes, the topological spinor, the 
incidence matrices, and the topological Dirac operator. 
Section \ref{sec:problem} introduces the DOSC framework, states the standing 
assumptions, formulates the optimization problem, derives the 
component-wise dynamics, and establishes the synchronized optimum as the 
global minimizer. Section \ref{sec:convergence} develops the convergence 
analysis through linearization, the isotropic stability condition, and a 
global Lyapunov argument. Section \ref{sec:numerical} reports numerical 
experiments on three admissible complexes of increasing complexity. 
Section \ref{sec:summary} closes the paper with a summary and a discussion of 
open problems.

%======================================================================
\section{Preliminaries}
\label{sec:preliminaries}
%======================================================================

In this section, we introduce the fundamental algebraic and topological 
objects that underpin our distributed optimization framework. To model 
interactions beyond pairwise relationships, we work on finite oriented 
complexes of dimension~$K=2$. This restriction simplifies the notation 
without loss of generality: all results extend to arbitrary~$K$ at the 
cost of additional indices.

\begin{definition}[Simplicial and cell complexes]
\label{def:complex}
Let $\mathcal{K}$ be a finite oriented complex of dimension~$K=2$. 
The complex is partitioned by dimension into sets of $k$-cells, 
denoted by~$\mathcal{K}_k$. Specifically, $\mathcal{K}_0$ is the set 
of $0$-cells (nodes) with cardinality~$N_0$, $\mathcal{K}_1$ is the 
set of $1$-cells (edges) with cardinality~$N_1$, and $\mathcal{K}_2$ 
is the set of $2$-cells (faces) with cardinality~$N_2$. The total 
number of cells is denoted by $\mathcal{N} = N_0 + N_1 + N_2$. 
\end{definition}

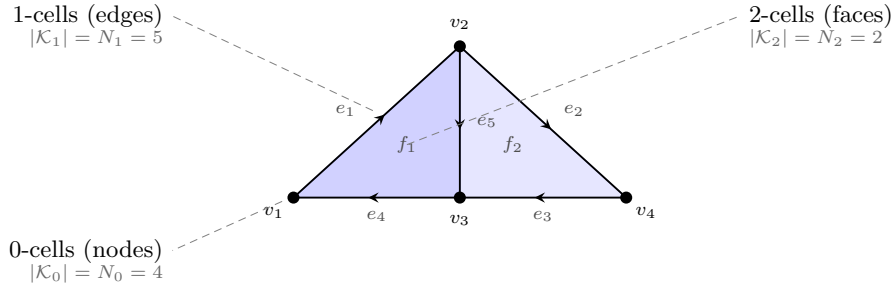
\begin{figure}[h]
\centering
\begin{tikzpicture}[
  scale=1.0,
  node/.style={circle,fill=black,inner sep=1.6pt},
  edgestyle/.style={line width=0.7pt,
                    decoration={markings,mark=at position 0.55 with 
                    {\arrow[line width=0.7pt]{stealth}}},
                    postaction={decorate}},
  face/.style={line width=0.4pt},
  facelabel/.style={font=\footnotesize\itshape,color=black!70},
  nodelabel/.style={font=\footnotesize},
  edgelabel/.style={font=\footnotesize,color=black!70},
  leader/.style={densely dashed,line width=0.3pt,color=black!50},
  label/.style={font=\small},
  sublabel/.style={font=\footnotesize,color=black!60}
]

% Faces (2-cells) -- drawn first so edges and nodes sit on top.
% Left face f1: oriented counter-clockwise (v1 -> v2 -> v3)
% Right face f2: oriented counter-clockwise (v2 -> v4 -> v3)
\fill[blue!18,draw=blue!55,face] (0,0) -- (2.2,2.0) -- (2.2,0) -- cycle;
\fill[blue!10,draw=blue!55,face] (2.2,2.0) -- (4.4,0) -- (2.2,0) -- cycle;

% Face labels (interior, italic)
\node[facelabel] at (1.5,0.7)  {$f_1$};
\node[facelabel] at (2.9,0.7)  {$f_2$};

% Edges (1-cells) with orientation arrows
\draw[edgestyle] (0,0)     -- (2.2,2.0);  % e1: v1 -> v2
\draw[edgestyle] (2.2,2.0) -- (4.4,0);    % e2: v2 -> v4
\draw[edgestyle] (4.4,0)   -- (2.2,0);    % e3: v4 -> v3
\draw[edgestyle] (2.2,0)   -- (0,0);      % e4: v3 -> v1
\draw[edgestyle] (2.2,2.0) -- (2.2,0);    % e5: v2 -> v3 (shared)

% Edge labels
\node[edgelabel,anchor=east]      at (0.95,1.15) {$e_1$};
\node[edgelabel,anchor=west]      at (3.45,1.15) {$e_2$};
\node[edgelabel,anchor=north]     at (3.30,-0.05){$e_3$};
\node[edgelabel,anchor=north]     at (1.10,-0.05){$e_4$};
\node[edgelabel,anchor=west]      at (2.30,1.00) {$e_5$};

% Nodes (0-cells)
\node[node] at (0,0)     {};  % v1
\node[node] at (2.2,2.0) {};  % v2
\node[node] at (2.2,0)   {};  % v3
\node[node] at (4.4,0)   {};  % v4

% Node labels
\node[nodelabel,anchor=north east] at (0,0)     {$v_1$};
\node[nodelabel,anchor=south]      at (2.2,2.1) {$v_2$};
\node[nodelabel,anchor=north]      at (2.2,-0.1){$v_3$};
\node[nodelabel,anchor=north west] at (4.4,0)   {$v_4$};

% Leader lines to dimension labels
\draw[leader] (-0.15,-0.05) -- (-1.6,-0.7);
\draw[leader] (1.05,1.15)   -- (-1.6,2.4);
\draw[leader] (1.5,0.7)     -- (5.9,2.4);

% Dimension labels
\node[label,anchor=east]    at (-1.6,-0.7) {0-cells (nodes)};
\node[sublabel,anchor=east] at (-1.6,-1.0) {$|\mathcal{K}_0| = N_0 = 4$};

\node[label,anchor=east]    at (-1.6,2.4)  {1-cells (edges)};
\node[sublabel,anchor=east] at (-1.6,2.1)  {$|\mathcal{K}_1| = N_1 = 5$};

\node[label,anchor=west]    at (5.9,2.4)   {2-cells (faces)};
\node[sublabel,anchor=west] at (5.9,2.1)   {$|\mathcal{K}_2| = N_2 = 2$};

\end{tikzpicture}
\caption{A two-dimensional oriented simplicial complex with 
$N_0 = 4$ nodes, $N_1 = 5$ edges, and $N_2 = 2$ triangular faces. 
The two faces $f_1$ and $f_2$ share the edge $e_5$, and edge 
orientations are chosen so that the induced boundary orientations 
of $f_1$ and $f_2$ on $e_5$ are opposite.}
\label{fig:simplicial_complex}
\end{figure}

A graphical representation of Definition \ref{def:complex} is shown in~\ref{fig:simplicial_complex}.

\begin{definition}[Incidence matrices]
\label{def:incidence}
The topological structure and the orientation of the cells are 
encoded by the incidence matrices. For $1$-cells, after assigning 
an arbitrary orientation to each edge $e_j \in \mathcal{K}_1$, one 
end-point is labeled as the head and the other as the tail. The 
matrix $B_1 \in \mathbb{R}^{N_0 \times N_1}$ has entries
\begin{equation}
[B_1]_{i,j} =
\begin{cases}
+1, & \text{if node } i \text{ is the head of edge } e_j,\\
-1, & \text{if node } i \text{ is the tail of edge } e_j,\\
0,  & \text{otherwise.}
\end{cases}
\end{equation}
Similarly, the matrix $B_2 \in \mathbb{R}^{N_1 \times N_2}$ maps 
faces to their bounding edges: $[B_2]_{j,l} = +1$ if the orientation 
of edge~$e_j$ matches the induced boundary orientation of 
face~$f_l$, $-1$ if they are opposite, and~$0$ if~$e_j$ is not 
incident to~$f_l$.
\end{definition}

A fundamental property of these boundary operators is the exactness 
condition $B_1 B_2 = 0$, which generalizes to $B_k B_{k+1} = 0$ for 
arbitrary~$K$ and encodes the geometric fact that the boundary of a 
boundary vanishes.

\begin{definition}[Topological spinor]
\label{def:spinor}
To describe states distributed across the complex, we assign to each simplex  
(cell) a signal in $\mathbb{R}^d$, where $d \geq 1$ is fixed. The 
collective state is aggregated into a single block vector, the 
\emph{topological spinor} $\bm{z} \in \mathbb{R}^{\mathcal{N} d}$,
\begin{equation}
\label{eq:spinor_def}
\bm{z} = \begin{bmatrix} z_0 \\ z_1 \\ z_2 \end{bmatrix},
\end{equation}
where $z_0 \in \mathbb{R}^{N_0 d}$, $z_1 \in \mathbb{R}^{N_1 d}$, 
and $z_2 \in \mathbb{R}^{N_2 d}$ contain the signals on nodes, 
edges, and faces, respectively. The integers $N_0, N_1, N_2$ are in 
general distinct, while the per-cell signal dimension~$d$ is shared 
across all dimensions.
\end{definition}

Also, we denote by $z_k^{\sigma}$ the decision variables related to the $\sigma$th $k$-dimensional simplex (cell) of the higher-order network.  

\begin{definition}[Topological Dirac operator]
\label{def:dirac}
Information exchange across topological dimensions is governed by 
the topological Dirac operator $D \in \mathbb{R}^{\mathcal{N} \times \mathcal{N}}$, 
defined as the symmetric block matrix
\begin{equation}
\label{eq:dirac_def}
D = \begin{bmatrix} 
0          & B_1        & 0   \\
B_1^{\top} & 0          & B_2 \\
0          & B_2^{\top} & 0 
\end{bmatrix}.
\end{equation}
\end{definition}

The block structure of~$D$ encodes the coupling between adjacent 
dimensions: it maps node signals to edges through~$B_1$, edge 
signals to both nodes and faces through~$B_1^{\top}$ and~$B_2$, 
and face signals back to edges through~$B_2^{\top}$. When~$d>1$, 
$D$ acts on the spinor as $D \otimes I_d$; we suppress the 
Kronecker factor whenever no ambiguity arises. The square~$D^2$ is 
block-diagonal and recovers the Hodge Laplacians of each 
dimension~\cite{bianconi2021top}; we do not use this identity 
explicitly in what follows, as our entire development takes place 
within the Dirac framework.

To allow distinct coupling strengths along the two directions of 
information flow, we introduce scalar constants $\gamma_0, \gamma_1 > 0$ 
and define the weighted Dirac operator
\begin{equation}
\label{eq:dirac_slashed}
\slashed{\mathcal{D}} = 
\begin{bmatrix} 
0               & \gamma_0 B_1   & 0              \\ 
\gamma_1 B_1^{\top} & 0          & \gamma_0 B_2   \\ 
0               & \gamma_1 B_2^{\top} & 0 
\end{bmatrix}.
\end{equation}
Here $\gamma_0$ weights the \emph{boundary} maps $B_1$ and $B_2$ 
(transfer of information from higher- to lower-dimensional cells: 
edges to nodes and faces to edges), while $\gamma_1$ weights the 
\emph{coboundary} maps $B_1^{\top}$ and $B_2^{\top}$ (transfer from 
lower- to higher-dimensional cells: nodes to edges and edges to 
faces). Equivalently,
\begin{equation}
\label{eq:dirac_dirsplit}
\slashed{\mathcal{D}} = \gamma_0\, D^{\downarrow} + \gamma_1\, D^{\uparrow},
\qquad
D^{\downarrow} = 
\begin{bmatrix} 0 & B_1 & 0 \\ 0 & 0 & B_2 \\ 0 & 0 & 0 \end{bmatrix},
\qquad
D^{\uparrow} = \bigl(D^{\downarrow}\bigr)^{\top},
\end{equation}
so that $D = D^{\downarrow} + D^{\uparrow}$. The operator 
$\slashed{\mathcal{D}}$ is \emph{not} symmetric unless 
$\gamma_0 = \gamma_1$; its symmetric part is 
$\tfrac{1}{2}(\slashed{\mathcal{D}} + \slashed{\mathcal{D}}^{\top}) = 
\tfrac{\gamma_0+\gamma_1}{2}\,D$, and it reduces to~$D$ when 
$\gamma_0 = \gamma_1 = 1$. This is 
precisely the directional asymmetry $\gamma_0 \neq \gamma_1$ that 
produces the product $\gamma_0\gamma_1$ in the stability condition.

\begin{remark}
\label{rem:scalar_gamma}
We restrict attention to two scalar coupling constants 
$\gamma_0,\gamma_1$, one per direction of information flow, shared 
across both the node--edge and edge--face sectors. The general 
formulation in~\cite{carletti2025gtds} allows matrix-valued gamma 
operators acting separately on each block of the spinor, with 
distinct constants per sector and per direction; the present 
two-constant choice is the minimal one that retains the 
directional asymmetry responsible for the product stability 
condition while keeping the analysis transparent. In the notation 
of~\cite{carletti2025gtds}, our $\gamma_0$ and $\gamma_1$ play the 
roles of the boundary and coboundary coupling constants common to 
both sectors.
\end{remark}

%======================================================================
\section{The Topological Distributed Optimization Framework}
\label{sec:problem}
%======================================================================

In this section, we introduce the distributed optimization 
framework. We first state the standing assumptions on the local 
cost functions and the underlying complex, formulate the optimization problem, derive the cell-wise dynamics, and establish that the synchronized 
equilibrium of these dynamics coincides with the global minimizer.

\subsection{Standing Assumptions}
\label{subsec:assumptions}

Throughout the rest of the paper, we make the following three 
assumptions: the first two on the local cost functions and the third 
on the topology of the complex.

\begin{assumption}[Strong convexity]\label{ass:convex}
For every $k \in \{0,1,2\}$ and every simplex (cell) $\sigma \in \mathcal{K}_k$, 
the local cost function $F_k^{\sigma} : \mathbb{R}^d \to \mathbb{R}$ 
is twice continuously differentiable and $\nu$-strongly convex; that 
is, there exists a constant~$\nu > 0$ such that 
$\nabla^2 F_k^{\sigma}(z_k^{\sigma}) \succeq \nu I_d$ for all 
$x \in \mathbb{R}^d$.
\end{assumption}

\begin{assumption}[Topological admissibility]
\label{ass:topology}
There exists a vector 
$\widehat{U} \in \mathbb{R}^{\mathcal{N}}$ with 
$\widehat{U}^\sigma_{0} = 1$ for every node and 
$\widehat{U}^\sigma_{k} \in \{-1, +1\}$ for every cell of dimension 
$k \in \{1,2\}$, such that
\begin{equation}
\label{eq:admissibility}
D \widehat{U} = 0.
\end{equation}
\end{assumption}

\begin{assumption}[Orientation-equivariant costs with dimensional targets]\label{ass:even}
For every $k \in \{0,1,2\}$ and every cell $\sigma \in \mathcal{K}_k$, 
the local cost function takes the form
\begin{equation}
\label{eq:cost_form}
F_k^{\sigma}(z_k^{\sigma}) = h_k^\sigma(z_k^{\sigma}) - \widehat{U}^\sigma_{k}\, c_k^{\sigma\top} z_k^{\sigma},
\end{equation}
where $h_k^\sigma : \mathbb{R}^d \to \mathbb{R}$ is even 
($h_k^\sigma(-z_k^{\sigma}) = h_k^\sigma(z_k^{\sigma})$), $c_k^\sigma \in \mathbb{R}^d$ is a fixed 
\emph{dimensional target}, and $\widehat{U}^\sigma_{k} \in \{-1,+1\}$ 
is the orientation sign of cell~$\sigma$ from Assumption \ref{ass:topology}.
\end{assumption}

Assumption \ref{ass:convex} is standard in distributed optimization and ensures 
the existence and uniqueness of a global minimizer for each local 
cost. Assumption \ref{ass:even} is the structural requirement imposed by cochain 
dynamics. Because the signal on an oriented cell changes sign under 
reversal of orientation, the local dynamics must be \emph{equivariant} 
under that reversal: flipping both the orientation 
($\widehat{U}^\sigma_{k} \mapsto -\widehat{U}^\sigma_{k}$) and the 
signal ($z_k^{\sigma} \mapsto -z_k^{\sigma}$) must leave the reaction term unchanged up to 
sign. Equation~\eqref{eq:cost_form} is precisely the general cost for 
which this holds: its gradient
\begin{equation}
\label{eq:grad_equivariant}
\nabla F_k^{\sigma}(z_k^{\sigma}) = \nabla h_k^\sigma(z_k^{\sigma}) - \widehat{U}^\sigma_{k} c_k^\sigma
\end{equation}
satisfies the equivariance identity 
$\nabla F_k^{\sigma}(z_k^{\sigma})\big|_{\widehat{U}^\sigma_{k}} 
= -\,\nabla F_k^{\sigma}(-z_k^{\sigma})\big|_{-\widehat{U}^\sigma_{k}}$, 
since $\nabla h_k^\sigma$ is odd. The even core~$h_k^\sigma$ carries the conditioning 
of the problem (its Hessian is the local Hessian), while the linear 
term~$-\widehat{U}^\sigma_{k} c_k^{\sigma\top}z_k^{\sigma}$ tied to the orientation 
sign shifts the target away from the origin without breaking 
equivariance. The purely even case $c_k^\sigma = 0$ recovers costs minimized 
at the origin; quadratic cores $h_k^\sigma(z_k^{\sigma}) = \tfrac12 z_k^{\sigma\top}Q_k^\sigma z_k^{\sigma}$ with 
$Q_k^\sigma \succ 0$, and more generally $h_k^\sigma(z_k^{\sigma}) = \phi_k^\sigma(\|z_k^{\sigma}\|^2)$ with 
$\phi_k^\sigma$ convex and increasing, satisfy the requirement automatically. 
We discuss the role of the dimensional targets and the limitation to 
per-dimension (rather than per-cell) objectives. Assumption \ref{ass:topology} is the 
topological admissibility condition characterized 
in~\cite{carletti2025gtds}; it constrains the class of feasible 
complexes and is analyzed in detail in Apprendix~\ref{app:existence}.

For a synchronized state to be an equilibrium, the dimensional targets 
must be mutually consistent. We record this as a standing compatibility 
condition.

\begin{assumption}[Target compatibility]
\label{ass:compat}
There exists $z^{*} \in \mathbb{R}^d$ such that 
$\nabla h_k^\sigma(z^{*}) = c_k^\sigma$ for every $k \in \{0,1,2\}$.
\end{assumption}

Assumption \ref{ass:compat} requires the three-dimensional targets to share a 
common preimage~$z^{*}$ under the respective gradient maps; for 
quadratic cores $h_k^\sigma(x) = \tfrac12 z_k^{\sigma\top}Q_k^\sigma z_k^{\sigma}$ it reads 
$z^{*} = (Q_k^\sigma)^{-1} c_k^\sigma$ for all~$k$, which reduces to $c_0^\sigma = c_1^\sigma = c_2^\sigma$ 
when the Hessians coincide. The point~$z^{*}$ is the common optimal 
state toward which the cells synchronize, and it is in general nonzero; 
the choice $c_k^\sigma = 0$ returns $z^{*} = 0$. The case of genuinely 
per-cell targets, where no common~$z^{*}$ exists, is discussed as an 
extension in Section \ref{sec:summary}.

The weighted construction~\eqref{eq:dirac_slashed} is consistent with the 
admissibility condition~\eqref{eq:admissibility}, which is stated for the 
unweighted operator~$D$. We record this for later use.

\begin{lemma}[Admissibility transfers to the weighted operator]
\label{lem:admiss_transfer}
If $D\widehat{U} = 0$, then $\slashed{\mathcal{D}}\,\widehat{U} = 0$ for 
all $\gamma_0,\gamma_1 > 0$.
\end{lemma}

\begin{proof}
Write $\widehat{U} = (\widehat{u}_{0}, \widehat{u}_{1}, 
\widehat{u}_{2})$. The condition $D\widehat{U} = 0$ expands into 
$B_1 \widehat{u}_{1} = 0$, 
$B_1^{\top}\widehat{u}_{0} + B_2\widehat{u}_{2} = 0$, and 
$B_2^{\top}\widehat{u}_{1} = 0$. The two summands in the middle 
equation are orthogonal, since 
$\langle B_1^{\top}\widehat{u}_{0}, B_2\widehat{u}_{2}\rangle = 
(\widehat{u}_{0})^{\top} B_1 B_2 \widehat{u}_{2} = 0$ by 
exactness; hence $B_1^{\top}\widehat{u}_{0} = 0$ and 
$B_2\widehat{u}_{2} = 0$ separately. Substituting these four 
identities into~\ref{eq:dirac_slashed} gives, block by block, 
$\gamma_0 B_1\widehat{u}_{1} = 0$, 
$\gamma_1 B_1^{\top}\widehat{u}_{0} + \gamma_0 B_2\widehat{u}_{2} = 0$, 
and $\gamma_1 B_2^{\top}\widehat{u}_{1} = 0$, i.e.\ 
$\slashed{\mathcal{D}}\,\widehat{U} = 0$.
\end{proof}

\subsection{Problem Formulation}
\label{subsec:problem}

Consider a complex~$\mathcal{K}$ satisfying Assumption \ref{ass:topology}, 
together with a family of local cost functions 
$\{F^\sigma_{k}\}_{\sigma \in \mathcal{K}_k,\, k=0,1,2}$ satisfying 
Assumptions \ref{ass:convex} and \ref{ass:even}. The aggregate cost is
\begin{equation}
\label{eq:global_opt}
\bm{F}(\bm{z}) := \sum_{k=0}^{2} \sum_{\sigma \in \mathcal{K}_k} 
F^\sigma_{k}\!\bigl(z_k^{\sigma}\bigr),
\qquad
\bm{z} \in \mathbb{R}^{\mathcal{N} d},
\end{equation}
where $z_k^{\sigma} \in \mathbb{R}^d$ denotes the block of the 
spinor associated with cell~$\sigma$ of dimension~$k$. The cost is 
separable across cells: each summand depends only on the local signal 
at its own cell. By Assumption \ref{ass:convex}, $\bm{F}$ is $\nu$-strongly convex 
on $\mathbb{R}^{\mathcal{N} d}$ and therefore admits a unique 
minimizer; the coupling between cells will arise from the dynamics 
introduced in Subsection \ref{subsec:dynamics}, not from the cost itself. The interplay between the strong convexity, the orientation-equivariant 
structure, and the compatibility condition determines the minimizer 
explicitly, and shows that it is in general nonzero.

\begin{proposition}[Common nonzero optimal state]
\label{prop:origin}
Under Assumptions \ref{ass:convex} and \ref{ass:even}, let $z^{*}$ be the common 
preimage presented in Assumption \ref{ass:compat}. Then for every cell $\sigma$ of 
dimension~$k$ the local cost $F^\sigma_{k}$ attains its unique global 
minimum at $\widehat{U}^\sigma_{k} z^{*}$. Consequently, the spinor 
$\Phi^{*} = \widehat{U}\otimes z^{*}$, whose block on cell~$\sigma$ is 
$\widehat{U}^\sigma_{k} z^{*}$, is the unique global minimizer of the 
separable aggregate~\eqref{eq:global_opt}, and $z^{*} \neq 0$ whenever 
some target $c_k^\sigma \neq 0$.
\end{proposition}

\begin{proof}
By Assumption \ref{ass:convex}, $F^\sigma_{k}$ is strongly convex with a unique 
minimizer characterized by 
$\nabla F_k^{\sigma}(z_k^{\sigma}) = \nabla h_k^\sigma(z_k^{\sigma}) - \widehat{U}^\sigma_{k} c_k^\sigma 
= 0$. Setting $z_k^{\sigma} = \widehat{U}^\sigma_{k} z^{*}$ and using the oddness 
of $\nabla h_k^\sigma$ gives 
$\nabla h_k^\sigma(\widehat{U}^\sigma_{k} z^{*}) 
= \widehat{U}^\sigma_{k}\nabla h_k^\sigma(z^{*}) 
= \widehat{U}^\sigma_{k} c_k^\sigma$ by Assumption \ref{ass:compat}, so the gradient 
vanishes and $\widehat{U}^\sigma_{k} z^{*}$ is the minimizer of 
$F^\sigma_{k}$. Since $\bm{F}$ is separable, $\Phi^{*}$ minimizes each 
summand simultaneously and is therefore the unique global minimizer 
of~$\bm{F}$. If $h_k^\sigma$ is strongly convex and $c_k^\sigma \neq 0$, then 
$\nabla h_k^\sigma(z^{*}) = c_k^\sigma \neq 0 = \nabla h_k^\sigma(0)$ forces $z^{*} \neq 0$.
\end{proof}

In the language of distributed optimization, Proposition \ref{prop:origin} states 
that the cells seek consensus on a common optimal state~$z^{*}$ 
determined by the dimensional targets~$c_k^\sigma$ through 
$\nabla h_k^\sigma(z^{*}) = c_k^\sigma$; the orientation signs~$\widehat{U}^\sigma_{k}$ 
then propagate this state coherently across the complex as 
$\widehat{U}^\sigma_{k} z^{*}$. The purely even case $c_k^\sigma \equiv 0$ 
collapses the optimum to the origin, recovering a topological consensus 
problem; with nonzero targets, the framework solves a genuine 
optimization in which the optimal state is not known a priori but is 
the common solution of the three gradient equations. We retain the 
symbol $z^{*}$ for this common optimal state throughout.

This formulation generalizes our previous work, in which the sum ran 
over agents and layers in multiplex~\cite{rodriguezcamargo2023consensus} 
and multilayer~\cite{rodriguezcamargo2025} networks. Here the sum 
ranges over topological dimensions~$k$ and cells~$\sigma$, so that 
cost contributions arise from all structural levels of the complex 
simultaneously, and the dimensional targets~$c_k^\sigma$ play the role of the 
per-layer setpoints of the coordinated-dispatch 
model~\cite{rodriguezcamargo2023consensus}.

%\begin{remark}[Dimensional vs.\ per-cell targets, and the role of equivariance]
%\label{rem:even_discussion}
The structure~\eqref{eq:cost_form} is the most general one consistent 
with orientation equivariance: the even core~$h_k^\sigma$ may carry any 
$\nu$-strongly convex conditioning, and the linear term tied to 
$\widehat{U}^\sigma_{k}$ displaces the optimum to a nonzero~$z^{*}$ 
without breaking the sign symmetry that makes the synchronized state 
well-defined. What equivariance does forbid is a target that varies 
\emph{from cell to cell within a dimension}: a per-cell linear term 
$-c_k^{\sigma \top} z_k^{\sigma}$ with $c^\sigma_{k}$ not proportional to 
$\widehat{U}^\sigma_{k}$ destroys equivariance, and, more 
fundamentally, when the local minimizers $z^{\sigma *}$ differ across 
cells, no synchronized state can annihilate every local gradient 
simultaneously, so no exact synchronized equilibrium exists. In that 
regime, the natural object is the optimum \emph{restricted} to the 
synchronization manifold, 
$z^{*} = \arg\min_z \sum_{k,\sigma}F_k^{\sigma}(\widehat{U}^\sigma_{k} z)$, 
a consensus optimum balancing conflicting per-cell objectives; the 
DOSC dynamics then converge to a neighborhood of 
$\widehat{U}\otimes z^{*}$ whose size is governed by the dispersion of 
the local minimizers. A complete treatment of per-cell targets, 
including the projected dynamics that restore exact convergence, is 
deferred to Section \ref{sec:summary} as a direction for future work. 
Assumption \ref{ass:compat} keeps the present development in the regime where an 
exact, nonzero synchronized optimum exists.
%\end{remark}

\subsection{DOSC Dynamics}
\label{subsec:dynamics}

The topological distributed optimization dynamics couple local 
gradient descent on each cell with the weighted Dirac operator 
of~\eqref{eq:dirac_slashed}:
\begin{equation}
\label{eq:tdo}
\frac{d\bm{z}}{dt} = -\nabla \bm{F}(\bm{z}) - \slashed{\mathcal{D}} \, \bm{z},
\end{equation}
where 
$\nabla \bm{F}(\bm{z}) = \bigl(\nabla F_0(z_0)^{\top},\, \nabla F_1(z_1)^{\top},\, 
\nabla F_2(z_2)^{\top}\bigr)^{\top}$ 
is the stacked gradient, with 
$\bigl[\nabla F_k(z_k)\bigr]^\sigma = \nabla F_k^{\sigma}\!\bigl(z_k^{\sigma}\bigr)$ 
for each $\sigma \in \mathcal{K}_k$.

Written block by block, the dynamics~\eqref{eq:tdo} read:
\begin{align}
\frac{d z_0^{i}}{dt} 
&= -\nabla F^{i}_{0}\!\bigl(z_0^{i}\bigr) 
- \gamma_0 \!\!\sum_{e \in \mathcal{K}_1} \!\! [B_1]_{i,e} \, z_1^{e}, 
\label{eq:node_dyn}\\[0.25em]
\frac{d z_1^{j}}{dt} 
&= -\nabla F^{j}_{1}\!\bigl(z_1^{j}\bigr) 
- \gamma_1 \!\!\sum_{i \in \mathcal{K}_0} \!\! [B_1^{\top}]_{j,i} \, z_0^{i} 
- \gamma_0 \!\!\sum_{f \in \mathcal{K}_2} \!\! [B_2]_{j,f} \, z_2^{f}, 
\label{eq:edge_dyn}\\[0.25em]
\frac{d z_2^{l}}{dt} 
&= -\nabla F^{l}_{2}\!\bigl(z_2^{l}\bigr) 
- \gamma_1 \!\!\sum_{e \in \mathcal{K}_1} \!\! [B_2^{\top}]_{l,e} \, z_1^{e}.
\label{eq:face_dyn}
\end{align}
Each cell performs gradient descent on its own local objective and 
exchanges information only with adjacent cells through the weighted 
Dirac operator. The boundary coupling~$\gamma_0$ governs the influence 
of edges on nodes and of faces on edges; the coboundary 
coupling~$\gamma_1$ governs the influence of nodes on edges and of 
edges on faces. The dynamics are thus genuinely distributed: each 
node communicates with its incident edges, each edge with both its 
endpoint nodes and its incident faces, and each face with its 
bounding edges. This cross-dimensional information flow is the 
defining feature of the Dirac coupling and is absent from all 
Laplacian-based methods.

Since the right-hand side of~\eqref{eq:tdo} is the sum of a 
continuously differentiable (indeed smooth) gradient field and a 
linear operator, it is locally Lipschitz on~$\mathbb{R}^{\mathcal{N}d}$. 
Picard--Lindel\"of then guarantees a unique maximal solution from every 
initial condition; the global convergence established 
in Theorem~\ref{thm:global} below shows that this solution is bounded and 
hence defined for all $t \geq 0$.

\subsection{The synchronized optimum}
\label{subsec:sync_opt}

The global minimizer of~\eqref{eq:global_opt} interacts with the 
admissibility vector $\widehat{U}$ of Assumption~\ref{ass:topology} in a clean 
way: the spinor in which every cell carries the common optimal 
state~$z^{*}$, modulated by the orientation signs~$\widehat{U}$, is the 
unique synchronized equilibrium of the DOSC dynamics and coincides with 
the global optimum.

\begin{theorem}[Synchronized equilibrium as global optimum]
\label{thm:sync_opt}
Under Assumptions \ref{ass:convex}-\ref{ass:compat}, let $z^{*}$ be 
the common optimal state of Proposition \ref{prop:origin}. Then the spinor
\begin{equation}
\label{eq:sync_opt_state}
\Phi^{*} = \widehat{U} \otimes z^{*}
\end{equation}
is the unique equilibrium of the DOSC dynamics~\eqref{eq:tdo} of 
synchronized form $\widehat{U}\otimes z$, and it coincides with the 
global minimizer of the aggregate cost~\eqref{eq:global_opt}. The 
optimal state~$z^{*}$ is in general nonzero.
\end{theorem}

\begin{proof}
We evaluate the right-hand side of~\eqref{eq:tdo} at a synchronized 
state $\widehat{U}\otimes v$. \emph{Dirac coupling term.} By bilinearity of the Kronecker product and 
Lemma \ref{lem:admiss_transfer},
\begin{equation*}
\slashed{\mathcal{D}}\,(\widehat{U}\otimes v) 
= \bigl(\slashed{\mathcal{D}}\,\widehat{U}\bigr) \otimes v = 0
\qquad \text{for every } v \in \mathbb{R}^d.
\end{equation*}
Hence the coupling vanishes along the entire synchronization manifold 
$\mathcal{S} = \{\widehat{U}\otimes v: v\in\mathbb{R}^d\}$, which is 
therefore invariant under~\eqref{eq:tdo}.

\emph{Gradient term.} For each cell $\sigma$ of dimension~$k$, the block 
of $\widehat{U}\otimes v$ is $\widehat{U}^\sigma_{k} v$, and using the 
cost form~\eqref{eq:cost_form} together with the oddness of $\nabla h_k^\sigma$,
\begin{equation*}
\nabla F_k^{\sigma}\!\bigl(\widehat{U}^\sigma_{k} v\bigr) 
= \nabla h_k^\sigma\!\bigl(\widehat{U}^\sigma_{k} v\bigr) 
- \widehat{U}^\sigma_{k} c_k^\sigma
= \widehat{U}^\sigma_{k}\bigl(\nabla h_k^\sigma(v) - c_k^\sigma\bigr),
\end{equation*}
where we used $(\widehat{U}^\sigma{k})^2 = 1$. A synchronized state 
is therefore an equilibrium if and only if $\nabla h_k^\sigma(v) = c_k^\sigma$ for 
every dimension~$k$, i.e.\ $v$ is a common preimage of the dimensional 
targets. By Assumption \ref{ass:compat} and the strong convexity of the 
cores~$h_k^\sigma$, the unique such point is $v = z^{*}$, so 
$\Phi^{*} = \widehat{U}\otimes z^{*}$ is the unique synchronized 
equilibrium.

\emph{Optimality.} By Proposition \ref{prop:origin} the block of $\Phi^{*}$ on 
cell~$\sigma$, namely $\widehat{U}^\sigma_{k} z^{*}$, is the minimizer 
of $g_\sigma^{(k)}$; hence $\Phi^{*}$ minimizes every summand 
of~\eqref{eq:global_opt} simultaneously and is the global minimizer 
of~$\bm{F}$, unique by strong convexity. By Proposition \ref{prop:origin}, 
$z^{*}\neq 0$ whenever some target $c_k^\sigma\neq 0$.
\end{proof}

%\begin{remark}[Role of admissibility]
%\label{rem:dirac_vanishes}
Theorem \ref{thm:sync_opt} clarifies the role of the admissibility condition. 
The Dirac coupling vanishes \emph{along the whole synchronization 
manifold}~$\mathcal{S}$, not merely at the optimum, so that 
$\mathcal{S}$ is dynamically invariant, and the cells can evolve 
coherently---following the common reduced flow 
$\dot{\bm{z}} = -\nabla \bm{F}(\bm{z})$ when the local costs are identical---toward the 
optimum. This is the topological feasibility content of 
Assumption \ref{ass:topology}: only complexes admitting a constant-modulus 
harmonic vector $\widehat{U}\in\ker D$ support such a synchronized 
evolution. The condition has no analogue in Laplacian-based distributed 
optimization, where the consensus subspace is always available on a 
connected graph. As in the Laplacian case, the coupling vanishes at the 
synchronized state and serves only as a driving force; the distinguishing 
feature here is that it acts across topological dimensions rather than 
within a single one.
%\end{remark}

%\begin{remark}[Synchronization as a computational mechanism vs.\ a 
%topological constraint]
%\label{rem:interpretation}
Two readings of the framework are worth separating. As a 
\emph{computational mechanism}, the Dirac-coupled gradient flow solves 
the consensus optimization problem ``minimize the aggregate cost over 
the synchronization manifold'' without dual variables. As a 
\emph{topologically constrained problem}, the feasible set is itself 
dictated by the algebraic topology of the complex through 
Assumption \ref{ass:topology}: the synchronized optimum exists as a coherent 
dynamical state only on admissible complexes. The two readings are 
complementary, and Theorem \ref{thm:sync_opt} shows they are consistent: on an 
admissible complex, the mechanism converges to the topologically 
admissible optimum.
%\end{remark}

%======================================================================
\section{Convergence analysis}
\label{sec:convergence}
%======================================================================

In this section, we analyze the convergence of the DOSC 
dynamics~\eqref{eq:tdo} to the synchronized optimum 
$\Phi^{*} = \widehat{U} \otimes z^{*}$ established in 
Theorem \ref{thm:sync_opt}. We first linearize the dynamics around 
$\Phi^{*}$ and exploit the Dirac decomposition to separate the 
perturbation into harmonic and non-harmonic components. We then derive an explicit stability 
condition in the isotropic case via a Master Stability Function 
analysis, and finally establish global 
convergence through a weighted Lyapunov argument that controls the 
indefinite, non-normal Dirac coupling.

\subsection{Linearization and Dirac decomposition}
\label{subsec:linearization}

Let $\delta \bm{z} := \bm{z} - \Phi^{*}$ denote the perturbation from the 
synchronized optimum. Linearizing~\eqref{eq:tdo} around $\Phi^{*}$ 
yields
\begin{equation}
\label{eq:linearized}
\frac{d \delta \bm{z}}{dt} = -\bigl( H_F + \slashed{\mathcal{D}} \bigr) 
\delta \bm{z},
\end{equation}
where 
$H_F = \diag\bigl( \nabla^2F_k^{\sigma}(\widehat{U}^\sigma_{k} z^{*}) 
\bigr)_{\sigma, k}$ 
is the block-diagonal Hessian matrix evaluated at the synchronized 
optimum. The linear term in~\eqref{eq:cost_form} does not contribute to 
the Hessian, so $\nabla^2F_k^{\sigma} = \nabla^2 h_k^\sigma$ and the 
stability analysis depends only on the even cores~$h_k^\sigma$; the dimensional 
targets~$c_k^\sigma$ relocate the equilibrium but leave its local stability 
unchanged. Because $\Phi^{*}$ is an equilibrium and not a limit cycle, 
the matrix $H_F + \slashed{\mathcal{D}}$ is constant in time, so 
stability is governed directly by its eigenvalues; no Floquet theory 
is required.

The analysis of~\eqref{eq:linearized} is simplified by decomposing 
the perturbation according to the topological structure of the Dirac 
operator. We use the sector operators
\begin{equation}
\label{eq:sector_ops}
D_{[1]} = \begin{bmatrix}
0 & B_1 & 0 \\
B_1^{\top} & 0 & 0 \\
0 & 0 & 0
\end{bmatrix},
\qquad
D_{[2]} = \begin{bmatrix}
0 & 0 & 0 \\
0 & 0 & B_2 \\
0 & B_2^{\top} & 0
\end{bmatrix},
\end{equation}
so that $D = D_{[1]} + D_{[2]}$. These are the unweighted node--edge and 
edge--face blocks of the Dirac operator; they should not be confused 
with the directional operators $D^{\downarrow}, D^{\uparrow}$ 
of~\eqref{eq:dirac_dirsplit}, which split $\slashed{\mathcal{D}}$ by 
direction of flow rather than by sector.

\begin{lemma}[Dirac decomposition and invariance]
\label{lem:dirac_decomp}
Let $D_{[1]}, D_{[2]}$ be as in~\eqref{eq:sector_ops}. Then
\begin{enumerate}
\item[(i)] $D_{[1]} D_{[2]} = D_{[2]} D_{[1]} = 0$;
\item[(ii)] the subspaces $\operatorname{im}(D_{[1]})$, 
$\operatorname{im}(D_{[2]})$, and $\ker(D)$ are mutually orthogonal and 
span $\mathbb{R}^{\mathcal{N}d}$, so that every spinor admits a unique 
decomposition
\begin{equation}
\label{eq:dirac_decomp}
\delta \bm{z} = \delta \bm{z}_{[1]} + \delta \bm{z}_{[2]} + \delta \bm{z}^{\mathrm{harm}},
\qquad
\delta \bm{z}_{[\cdot]}\in\operatorname{im}(D_{[\cdot]}),\ 
\delta \bm{z}^{\mathrm{harm}}\in\ker(D);
\end{equation}
\item[(iii)] the weighted operator $\slashed{\mathcal{D}}$ leaves each of the 
three subspaces invariant and annihilates $\ker(D)$, for all 
$\gamma_0,\gamma_1$.
\end{enumerate}
\end{lemma}

\begin{proof}
(i) A direct block computation gives $D_{[1]}D_{[2]} = 0$ with the only 
potentially non-zero block being $B_1 B_2$ (a node--face term), which 
vanishes by exactness; $D_{[2]}D_{[1]} = 0$ follows by transposition 
since $D_{[k]}^{\top} = D_{[k]}$.

(ii) The operators $D_{[1]}, D_{[2]}$ are symmetric, so for any 
$x, y$, $\langle D_{[1]} x, D_{[2]} y\rangle = x^{\top} D_{[1]} D_{[2]} y 
= 0$ by~(i); thus $\operatorname{im}(D_{[1]})\perp 
\operatorname{im}(D_{[2]})$. For a symmetric operator, the image is the 
orthogonal complement of the kernel; since $\ker(D) = \ker(D_{[1]})\cap 
\ker(D_{[2]})$ (a spinor is annihilated by $D$ iff it is annihilated by 
both sectors, again using exactness), $\ker(D)$ is orthogonal to both 
images. A dimension count via the rank--nullity theorem shows the three 
subspaces span the whole space, giving uniqueness. This is the Dirac 
analogue of the Hodge decomposition 
$C^{n} = \operatorname{im}(B_{n}^{\top})\oplus\ker(L_{[n]})\oplus 
\operatorname{im}(B_{n+1})$ into gradient, harmonic, and curl 
components~\cite{bianconi2021top}: $\operatorname{im}(D_{[1]})$ collects 
the gradient (node--edge) content, $\operatorname{im}(D_{[2]})$ the curl 
(edge--face) content, and $\ker(D)$ the harmonic content.

(iii) Let $\delta \bm{z}_{[1]} = (u, v, 0)$ with $u\in\operatorname{im}(B_1)$, 
$v\in\operatorname{im}(B_1^{\top})$. Its face block under 
$\slashed{\mathcal{D}}$ is $\gamma_1 B_2^{\top} v$, which vanishes 
because $v\in\operatorname{im}(B_1^{\top})$ and $B_2^{\top}B_1^{\top} = 
(B_1 B_2)^{\top} = 0$; the remaining blocks $\gamma_0 B_1 v$ and 
$\gamma_1 B_1^{\top} u$ lie in $\operatorname{im}(B_1)$ and 
$\operatorname{im}(B_1^{\top})$, respectively, so 
$\slashed{\mathcal{D}}\delta \bm{z}_{[1]} \in\operatorname{im}(D_{[1]})$. The 
argument for $\operatorname{im}(D_{[2]})$ is symmetric, using 
$B_1 B_2 = 0$. Finally, if $\delta \bm{z}^{\mathrm{harm}}\in\ker(D)$, then as 
in the proof of Lemma \ref{lem:admiss_transfer} each of 
$B_1 v$, $B_2^{\top} v$, $B_1^{\top} u$, $B_2 w$ vanishes, so 
$\slashed{\mathcal{D}}\delta \bm{z}^{\mathrm{harm}} = 0$.
\end{proof}

Applying the decomposition~\eqref{eq:dirac_decomp} 
to~\eqref{eq:linearized}, the harmonic component evolves under the 
Hessian alone, since $\slashed{\mathcal{D}} \delta \bm{z}^{\mathrm{harm}} = 0$:
\begin{equation}
\label{eq:harmonic_dyn}
\frac{d \delta \bm{z}^{\mathrm{harm}}}{dt} 
= - H_F \, \delta \bm{z}^{\mathrm{harm}}.
\end{equation}
By Assumption \ref{ass:convex}, $H_G \succeq \nu I_{\mathcal{N} d}$, so 
$\delta \bm{z}^{\mathrm{harm}}$ decays exponentially at rate at 
least~$\nu$ regardless of the coupling parameters. The 
non-harmonic components $\delta \bm{z}_{[1]}, \delta \bm{z}_{[2]}$ are 
governed by the singular values of $B_1$ and $B_2$, which we analyze 
next.

\subsection{Isotropic stability condition}
\label{subsec:isotropic}

We now derive an explicit stability criterion under the additional 
assumption that all local Hessians coincide at the synchronized 
optimum.

\begin{theorem}[Local stability, isotropic case]
\label{thm:isotropic}
Under Assumptions \ref{ass:convex}-\ref{ass:compat} and that 
$\nabla^2F_k^{\sigma}(\widehat{U}^\sigma_{k} z^{*}) = \nu I_d$ for 
every cell $\sigma$ and every $k \in \{0,1,2\}$. The synchronized optimum 
$\Phi^{*} = \widehat{U} \otimes z^{*}$ is locally exponentially 
stable under the DOSC dynamics~\eqref{eq:tdo} if and only if
\begin{equation}
\label{eq:stability_condition}
\gamma_0 \gamma_1 < \frac{\nu^2}{b_{\max}^2},
\end{equation}
where $b_{\max} = \max\bigl\{ \sigma_{\max}(B_1),\, \sigma_{\max}(B_2) \bigr\}$ 
is the spectral radius of the complex. Moreover, the slowest 
convergence rate is $\nu - b_{\max} \sqrt{\gamma_0 \gamma_1}$.
\end{theorem}

\begin{proof}
Under the isotropic assumption, $H_G = \nu I_{\mathcal{N} d}$ 
and~\eqref{eq:linearized} becomes 
$d \delta \bm{z} / dt = -(\nu I + \slashed{\mathcal{D}}) \delta \bm{z}$, so 
stability is determined by the spectrum of $\slashed{\mathcal{D}}$.

By Lemma \ref{lem:dirac_decomp}(iii), $\slashed{\mathcal{D}}$ leaves 
$\operatorname{im}(D_{[1]})$, $\operatorname{im}(D_{[2]})$, and 
$\ker(D)$ invariant, so its spectrum is the union of its spectra on the 
three subspaces. On $\ker(D)$ it is zero. On 
$\operatorname{im}(D_{[1]})$, let $b_\alpha$ be a non-zero singular 
value of $B_1$ with left and right singular vectors $u_\alpha, 
v_\alpha$, so $B_1 v_\alpha = b_\alpha u_\alpha$ and 
$B_1^{\top} u_\alpha = b_\alpha v_\alpha$. The two-dimensional subspace 
$\mathcal{S}_\alpha = \operatorname{span}\{(u_\alpha,0,0),(0,v_\alpha,0)\}$ 
is invariant, and in this basis
\begin{equation}
\label{eq:block_ne}
\slashed{\mathcal{D}}\big|_{\mathcal{S}_\alpha} 
= \begin{bmatrix} 0 & \gamma_0 b_\alpha \\ \gamma_1 b_\alpha & 0 \end{bmatrix},
\end{equation}
whose eigenvalues are $\pm b_\alpha\sqrt{\gamma_0\gamma_1}$. It is the 
directional asymmetry $\gamma_0\neq\gamma_1$ in the off-diagonal entries 
of~\eqref{eq:block_ne} that produces the geometric mean 
$\sqrt{\gamma_0\gamma_1}$; a sector-symmetric weighting would yield 
$\pm\gamma_0 b_\alpha$ instead, with no product. An identical reduction 
on $\operatorname{im}(D_{[2]})$ yields eigenvalues 
$\pm b_\alpha\sqrt{\gamma_0\gamma_1}$ for each non-zero singular value 
$b_\alpha$ of $B_2$. The full construction is given 
in Apprendix~\ref{app:msf}.

The eigenvalues of $-(\nu I + \slashed{\mathcal{D}})$ are therefore
\begin{equation}
\label{eq:eigenvalues}
\lambda_{\pm}^{(\alpha)} 
= -\nu \mp b_\alpha \sqrt{\gamma_0 \gamma_1},
\end{equation}
together with the eigenvalue $-\nu$ on $\ker(D)$. All eigenvalues have 
strictly negative real part if and only if 
$b_\alpha \sqrt{\gamma_0 \gamma_1} < \nu$ for every non-zero singular 
value, the binding constraint being $b_{\max}$, which 
gives~\eqref{eq:stability_condition}. The slowest non-trivial mode has 
rate $\nu - b_{\max}\sqrt{\gamma_0\gamma_1}$.
\end{proof}

Condition~\eqref{eq:stability_condition} is the Master Stability 
Function criterion specialized to the isotropic case. Its compact 
form reflects three independent design parameters: the conditioning 
of the optimization problem (through~$\nu$), the topological 
spectral content of the complex (through~$b_{\max}$), and the 
coupling strengths (through~$\gamma_0, \gamma_1$). The product 
structure~$\gamma_0 \gamma_1$ shows that the two coupling directions 
enter symmetrically in the isotropic regime, even though the operator 
$\slashed{\mathcal{D}}$ itself is not symmetric.

For heterogeneous Hessians $\nabla^2F_k^{\sigma}(z^{*}) \neq \nu I_d$, 
the spectral analysis above does not factorize cleanly, and the 
stability condition becomes a non-trivial function of the singular 
values $\{ b_\alpha \}$ and the local Hessian eigenvalues. The 
Routh--Hurwitz criterion applied to the characteristic polynomial of 
$-(H_F + \slashed{\mathcal{D}})$ restricted to each non-harmonic 
eigenspace provides sufficient algebraic conditions in this setting. 
We do not pursue the general analysis here; the global convergence 
result of the next section holds for arbitrary 
$\nu$-strongly convex local objectives, and the numerical experiments 
in Section~\ref{sec:numerical} illustrate the qualitative effect of 
heterogeneity on the convergence landscape.

\subsection{Global convergence}
\label{subsec:global}

The local analysis of Section~\ref {subsec:isotropic} establishes 
exponential stability in a neighborhood of $\Phi^{*}$. We now show 
that, under the same stability condition, convergence is global and 
occurs at the same sharp rate. The key device is a weighted Lyapunov 
function that symmetrizes the non-normal Dirac coupling.

\begin{theorem}[Global convergence]
\label{thm:global}
Under Assumption \ref{ass:convex}-\ref{ass:compat} and the stability 
condition~\eqref{eq:stability_condition}, the DOSC 
dynamics~\eqref{eq:tdo} converge globally exponentially to 
$\Phi^{*} = \widehat{U} \otimes z^{*}$ from any initial 
condition~$\bm{z}(0) \in \mathbb{R}^{\mathcal{N} d}$, at rate 
$\kappa = \nu - b_{\max}\sqrt{\gamma_0\gamma_1}$.
\end{theorem}

\begin{proof}
Introduce the block-diagonal positive definite weight
\begin{equation}
\label{eq:weight_P}
P = \diag\!\left( \tfrac{\gamma_1}{\gamma_0}\, I_{N_0 d},\; 
I_{N_1 d},\; \tfrac{\gamma_0}{\gamma_1}\, I_{N_2 d} \right),
\qquad
S = P^{-1/2} = \diag\!\left( \sqrt{\tfrac{\gamma_0}{\gamma_1}}\, I_{N_0 d},\; 
I_{N_1 d},\; \sqrt{\tfrac{\gamma_1}{\gamma_0}}\, I_{N_2 d} \right),
\end{equation}
and the weighted Lyapunov function
\begin{equation}
\label{eq:lyapunov}
V_P(\bm{z}) = \tfrac{1}{2}\, (\bm{z} - \Phi^{*})^{\top} P\, (\bm{z} - \Phi^{*}).
\end{equation}
Differentiating along trajectories of~\eqref{eq:tdo} and 
using~$\slashed{\mathcal{D}}\Phi^{*}=0$ (Lemma \ref{lem:admiss_transfer}),
\begin{equation}
\label{eq:Vdot}
\dot{V}_P 
= -(\bm{z} - \Phi^{*})^{\top} P\, \nabla \bm{F}(\bm{z}) 
\, - \, (\bm{z} - \Phi^{*})^{\top} P\, \slashed{\mathcal{D}}\, (\bm{z} - \Phi^{*}),
\end{equation}
which we bound term by term.

\emph{Gradient term.} Since $P$ acts as the positive scalar 
$p_k\in\{\gamma_1/\gamma_0,1,\gamma_0/\gamma_1\}$ on each dimension-$k$ 
block, and by Proposition \ref{prop:origin} each local cost is minimized at the 
block value $\widehat{U}^\sigma_{k}z^{*}$ (so that 
$\nabla F_k^{\sigma}(\widehat{U}^\sigma_{k}z^{*}) = 0$), the 
strong-convexity inequality~\cite{nesterov2018lectures} gives, cell by 
cell,
\begin{equation*}
\bigl( \bm{z}^\sigma_{k} - \widehat{U}^\sigma_{k} z^{*} \bigr)^{\top} 
\nabla F_k^{\sigma}\!\bigl(\bm{z}^\sigma_{k}\bigr)
\geq \nu \, \bigl\| \bm{z}^\sigma_{k} - \widehat{U}^\sigma_{k} z^{*} \bigr\|^2.
\end{equation*}
Multiplying by $p_k > 0$ and summing over all cells,
\begin{equation}
\label{eq:grad_bound}
(\bm{z} - \Phi^{*})^{\top} P\, \nabla \bm{F}(\bm{z}) 
\geq \nu \, (\bm{z} - \Phi^{*})^{\top} P\, (\bm{z} - \Phi^{*}) = 2\nu\, V_P.
\end{equation}

\emph{Dirac term.} Writing $\bm{x} = \bm{z} - \Phi^{*}$ and $\bm{y} = S^{-1} \bm{x}$, and 
using that $S$ is symmetric,
\begin{equation*}
\bm{x}^{\top} P\, \slashed{\mathcal{D}}\, \bm{x} 
= \bm{x}^{\top} S^{-2}\, \slashed{\mathcal{D}}\, \bm{x} 
= \bm{y}^{\top} \bigl( S^{-1}\slashed{\mathcal{D}}\,S \bigr) \bm{y} 
= \bm{y}^{\top} \widetilde{\slashed{\mathcal{D}}}\, \bm{y}.
\end{equation*}

The choice of weights in~\eqref{eq:weight_P} symmetrizes each 
block~\eqref{eq:block_ne}: the similarity $S^{-1}\slashed{\mathcal{D}}S$ 
turns the off-diagonal pair $(\gamma_0 b_\alpha,\gamma_1 b_\alpha)$ into 
$(b_\alpha\sqrt{\gamma_0\gamma_1},\, b_\alpha\sqrt{\gamma_0\gamma_1})$, so 
that $\widetilde{\slashed{\mathcal{D}}} = S^{-1}\slashed{\mathcal{D}}S$ 
is symmetric with eigenvalues $\pm b_\alpha\sqrt{\gamma_0\gamma_1}$ and 
spectral radius $b_{\max}\sqrt{\gamma_0\gamma_1}$ (the same spectrum as 
in Theorem~\ref{thm:isotropic}, since similarity preserves eigenvalues). 
Therefore
\begin{equation}
\label{eq:dirac_quadratic_bound}
\bigl| \bm{x}^{\top} P\, \slashed{\mathcal{D}}\, \bm{x} \bigr| 
= \bigl| \bm{y}^{\top}\widetilde{\slashed{\mathcal{D}}}\, \bm{y} \bigr|
\leq b_{\max}\sqrt{\gamma_0\gamma_1}\, \|\bm{y}\|^2 
= 2\, b_{\max}\sqrt{\gamma_0\gamma_1}\, V_P .
\end{equation}

\emph{Combining.} From~\eqref{eq:grad_bound} and \eqref{eq:dirac_quadratic_bound},
\begin{equation}
\label{eq:Vdot_final}
\dot{V}_P 
\leq -2\bigl( \nu - b_{\max}\sqrt{\gamma_0\gamma_1} \bigr) V_P 
= -2\kappa\, V_P,
\end{equation}
with $\kappa = \nu - b_{\max}\sqrt{\gamma_0\gamma_1} > 0$ under 
the stability condition~\eqref{eq:stability_condition}. Gr\"onwall's 
inequality yields $V_P(s(t)) \leq V_P(s(0))\, e^{-2\kappa t}$, and since 
$P\succ 0$, this establishes global exponential convergence to 
$\Phi^{*}$ at rate~$\kappa$, matching the local rate 
of Theorem~\ref{thm:isotropic}.
\end{proof}

The Dirac coupling is indefinite and, for $\gamma_0\neq\gamma_1$, 
non-normal, so it can either accelerate or decelerate convergence 
depending on the alignment of the perturbation with the positive or 
negative eigenspaces of~$\slashed{\mathcal{D}}$. This is a departure 
from Laplacian-based distributed optimization, where the coupling 
operator is positive semidefinite and contributes directly to the 
dissipation~\cite{rodriguezcamargo2023consensus,rodriguezcamargo2025}. 
The naive Euclidean Lyapunov function $V(\bm{z}) = \tfrac12\|\bm{z}-\Phi^{*}\|^2$ 
sees only the symmetric part of $\slashed{\mathcal{D}}$, namely 
$\tfrac{\gamma_0+\gamma_1}{2}D$, and therefore certifies only the more 
conservative \emph{arithmetic-mean} condition 
$\tfrac{\gamma_0+\gamma_1}{2}\, b_{\max} < \nu$. Because the arithmetic 
mean dominates the geometric mean, this is strictly stronger than the 
sharp condition~\eqref{eq:stability_condition} unless 
$\gamma_0=\gamma_1$, in which case the two coincide. The weighted 
function~\eqref{eq:lyapunov} removes the gap by symmetrizing the 
coupling, recovering the geometric-mean threshold of the spectral 
analysis and confirming that the global rate equals the local one.

%======================================================================
\section{Numerical experiments}
\label{sec:numerical}
%======================================================================

We illustrate the DOSC framework on three complexes of increasing 
structural complexity: a 9-node Eulerian cycle ($K=1$), a square 
lattice tessellation of the 2-torus (SLTT, $K=2$) with homogeneous 
quadratic objectives, and the same SLTT topology with heterogeneous 
cost coefficients modeling a multi-resource infrastructure network. 
All three complexes satisfy the topological admissibility condition 
$D \widehat{U} = 0$ of Assumption \ref{ass:topology}: the Eulerian cycle 
admits a non-trivial admissibility vector because every node has 
even degree, while the SLTT admits one because of the regular 
boundary cancellation of its square faces~\cite{carletti2025gtds}. 
The restriction to these canonical admissible complexes is 
deliberate: relaxing the admissibility condition to more general 
higher-order networks---in particular, weighted simplicial complexes 
satisfying algebraic conditions of the type discussed 
in~\cite{carletti2025gtds}---constitutes an open problem. In all experiments, the DOSC dynamics~\eqref{eq:tdo} are integrated 
numerically using a fourth-order Runge--Kutta scheme with fixed step 
size, and we measure the consensus time $t_c$ as the earliest time 
at which $\| \bm{z}(t) - \Phi^{*} \| < 10^{-3}$.

\subsection{Eulerian cycle}
\label{subsec:num_cycle}

We consider a 9-node Eulerian cycle ($K=1$), with local quadratic 
objectives 
\begin{equation}
\label{eq:cycle_obj}
F^{\sigma}_{0}(\theta) = \theta^{\top} \theta, 
\qquad 
F^{e}_{1}(\phi) = \phi^{\top} \phi,
\end{equation}
for $\sigma \in \mathcal{K}_0$ and $e \in \mathcal{K}_1$, with 
scalar signals ($d=1$). The strong convexity parameter is $\nu = 2$ 
and the global minimizer is $z^{*} = 0$. We set the node--edge 
coupling to $\gamma_0 = 0.1$, while $\gamma_1$ is irrelevant for 
$K=1$ since there are no faces.

Fig. \ref{fig:loop} shows six snapshots of the topological signals on 
the cycle. Starting from random initial values in $[0,1]$, the 
Dirac coupling between nodes and incident edges, combined with 
local gradient descent, drives all signals toward the synchronized 
optimum $\Phi^{*} = 0$. The initial chromatic heterogeneity 
across nodes and edges narrows progressively, and by $t = 100$ the 
graph displays uniform amplitude across both topological 
dimensions.

\begin{figure*}[t]
\centering
\includegraphics[width=0.95\textwidth]{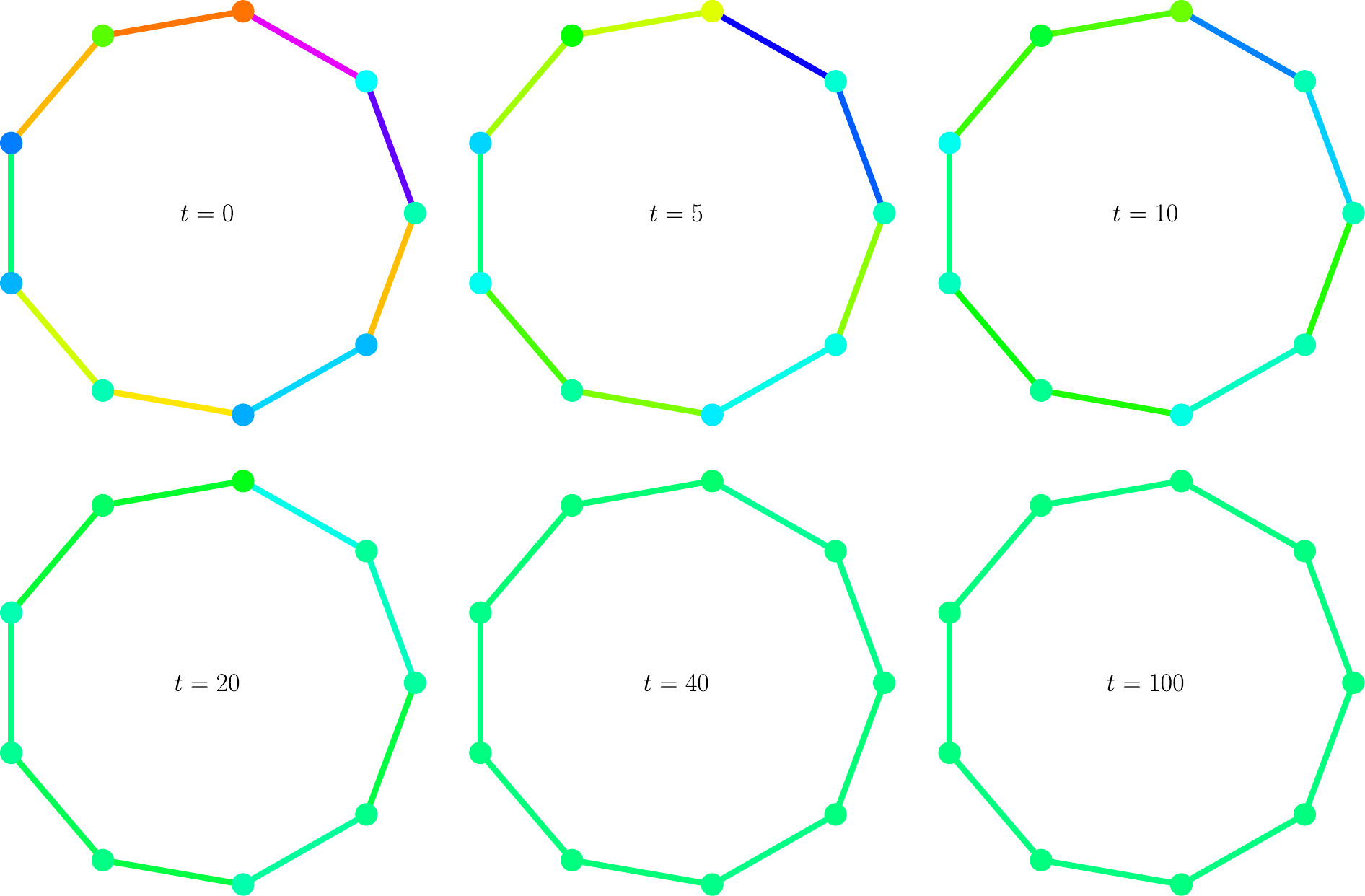}
\caption{Topological consensus on a 9-node Eulerian cycle 
($K=1$) under the DOSC dynamics~\eqref{eq:tdo} with 
$\gamma_0 = 0.1$ and $\gamma_1 = 0$. Snapshots at 
$t = 0, 5, 10, 20, 40, 100$ show the convergence of node and 
edge signals to the synchronized optimum 
$\Phi^{*} = \widehat{U} \otimes z^{*}$ with $z^{*} = 0$. Color 
encodes signal amplitude.}
\label{fig:loop}
\end{figure*}

The quantitative trajectories are shown 
in Fig.~\ref{fig:loop_evolution}: all node and edge signals converge 
monotonically to the common value~$z^{*} = 0$, with the edge 
signals exhibiting a wider initial spread and a faster initial 
transient. This asymmetry reflects the fact that each edge receives 
coupling from its two endpoint nodes through $B_1^{\top}$, while 
each cycle node is incident to only two edges. 
%The convergence is exponential in agreement with~\ref{thm:isotropic}, which predicts a slowest rate of $\nu - b_{\max} \sqrt{\gamma_0 \gamma_1} = \nu = 2$ when $\gamma_1 = 0$.

\begin{figure}[t]
\centering
\includegraphics[width=0.5\textwidth]{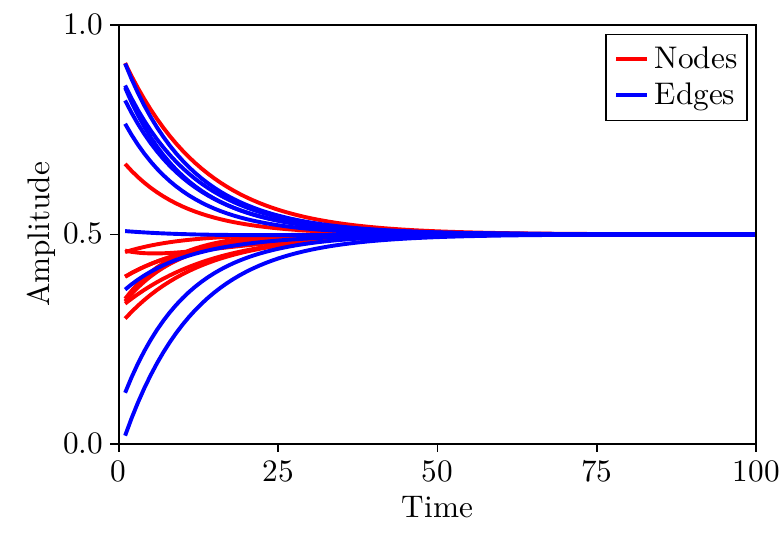}
\caption{Time evolution of all node signals (red) and edge 
signals (blue) for the Eulerian cycle of Fig.~\ref{fig:loop}. 
Starting from random initial values in~$[0,1]$, all signals 
converge exponentially to $z^{*} = 0$. The edge signals exhibit 
a wider initial spread due to the cross-dimensional Dirac 
coupling with incident nodes.}
\label{fig:loop_evolution}
\end{figure}

%\ref{fig:loop_consensus_time} reports the consensus time~$t_c$ 
%as a function of the coupling parameter~$\gamma_0$. The monotonic 
%increase of~$t_c$ with~$\gamma_0$ admits a clean interpretation: 
%since $\gamma_1 = 0$, the edge signals evolve under pure gradient 
%descent decoupled from the nodes, but the node dynamics receive a 
%perturbation $-\gamma_0 B_1 \phi$ from the still-evolving edge 
%signals. Larger~$\gamma_0$ injects larger perturbations from 
%not-yet-converged edges into the node dynamics, slowing the 
%overall relaxation. The sawtooth structure superimposed on the 
%trend reflects the discrete spectrum of the Dirac operator on the 
%cycle and the threshold-based definition of~$t_c$.

\begin{figure}[h]
  \includegraphics[width=\linewidth]{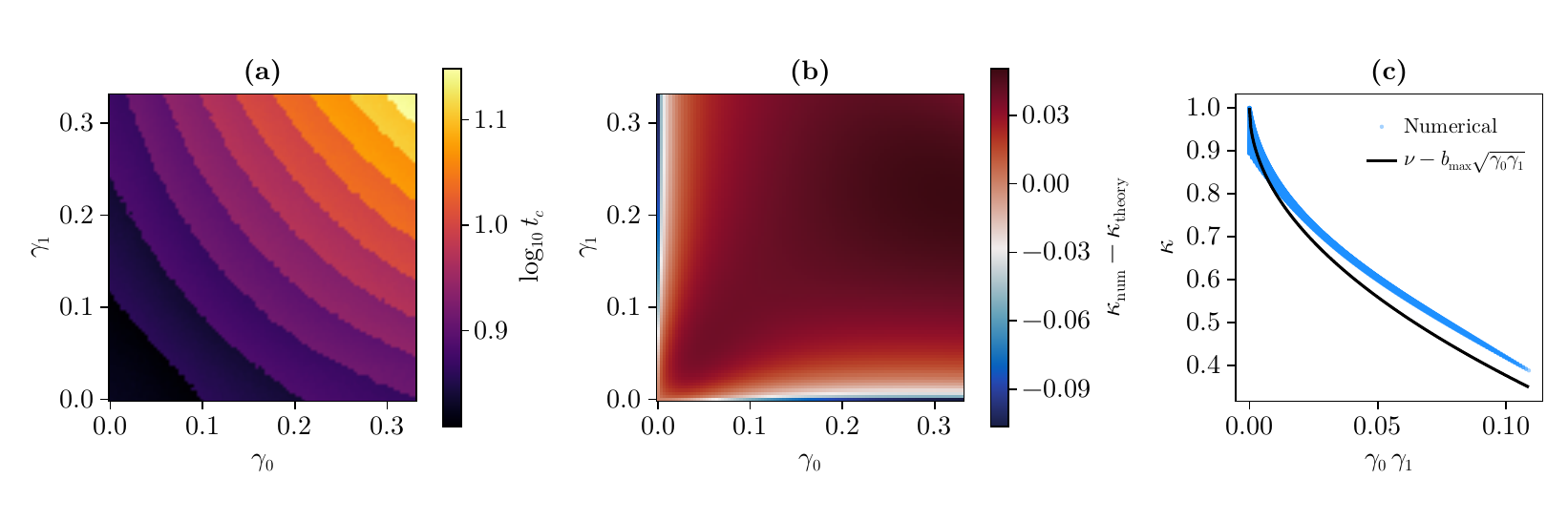}
  \caption{Convergence of the DOSC dynamics on the 9-node Eulerian cycle as a
  function of the directional couplings $\gamma_0$ and $\gamma_1$. (a)
  Logarithm of the consensus time. (b) Difference between the numerically
  measured asymptotic decay rate and the theoretical prediction of
  Theorem \ref{thm:isotropic}. (c) Numerical decay rate versus the product $\gamma_0
  \gamma_1$, with the theoretical curve (solid line).}
  \label{fig_loop}
\end{figure}

Fig.~\ref{fig_loop}(a) shows the consensus-time landscape: the $\log _{10} t_c$ contours
are symmetric about the diagonal $\gamma_0 = \gamma_1$ and follow the
hyperbolas of constant $\gamma_0 \gamma_1$ predicted by \eqref{eq:stability_condition}, the
numerical signature of the geometric-mean dependence. Fig.~\ref{fig_loop}(b) reports the
residual $\kappa_{\text{num}} -\kappa_{\text{theory}}$ between the numerical asymptotic
decay rate $\kappa_{\text{num}}$, estimated from the
slope of \ $\log  \parallel \bm{z} (t) - \Phi^{\ast} \parallel$ over the late-time
tail, and the isotropic prediction
$\kappa_{\text{theory}} = \nu - b_{\max } \sqrt{\gamma_0
\gamma_1}$. The residual stays within a few percent of $\nu$ across the
$(\gamma_0, \gamma_1)$ plane, confirming Theorem \ref{thm:isotropic}. Its small structured
sign is a finite-window artifact: in the interior the fit is biased upward by
faster modes (the harmonic component at rate $\nu$, \eqref{eq:stability_condition}, and blocks
with $b_{\alpha} < b_{\max }$), while along the axes, where $\gamma_0 \gamma_1
\rightarrow 0$, the Dirac block becomes defective and its algebraic prefactor
$t \text{} e^{- \nu t}$ depresses the slope below $\nu$. Fig.~\ref{fig_loop}(c) makes the underlying rate law explicit: the
asymptotic rate remains $\kappa$ throughout, with
$\kappa_{\text{num}}$ collapsing onto the single curve
$\nu - b_{\max } \sqrt{\gamma_0 \gamma_1}$ independently of the individual
couplings.

%\begin{figure}[t]
%\centering
%\includegraphics[width=0.5\textwidth]{consensus_time_loop.pdf}
%\caption{Consensus time $t_c$ (logarithmic scale) as a function 
%of the coupling parameter $\gamma_0$ for the 9-node Eulerian 
%cycle with $\gamma_1 = 0$. The consensus time increases 
%monotonically with $\gamma_0$, reflecting the perturbation that 
%not-yet-converged edge signals inject into the node dynamics.}
%\label{fig:loop_consensus_time}
%\end{figure}

\subsection{Homogeneous SLTT}
\label{subsec:num_sltt}

We now move to a $K=2$ cell complex: a square lattice tessellation 
of the 2-torus with $N_0 = 128$ nodes, $N_1 = 256$ edges, and 
$N_2 = 128$ faces. The local objectives are
\begin{equation}
\label{eq:sltt_obj}
F^{\sigma}_{k}(x) = \tfrac{1}{2} x^{\top} x,
\end{equation}
identical across all cells and dimensions, yielding 
$\nu = 1$ and $z^{*} = 0$. We set 
$\gamma_0 = \gamma_1 = 0.1$, with initial conditions sampled 
uniformly from~$[-5, 5]$. The spectral radius of the SLTT is 
$b_{\max} = 2\sqrt{2}$, so the stability 
condition~\eqref{eq:stability_condition} becomes 
$\gamma_0 \gamma_1 < 1/8$ and is satisfied by an order of 
magnitude.

Fig. \ref{fig:torus_homogeneous} displays snapshots of all three 
families of topological signals on the torus surface. Unlike the 
Eulerian cycle, the SLTT engages the full cross-dimensional 
information flow: nodes receive signals from edges through~$B_1$, 
edges from both nodes and faces through~$B_1^{\top}$ and~$B_2$, 
and faces from their bounding edges through~$B_2^{\top}$. The 
strongly heterogeneous initial pattern progressively homogenizes, 
and by~$t = 100$ all three topological dimensions have converged 
to the synchronized optimum.

\begin{figure*}[t]
\centering
\includegraphics[width=0.95\textwidth]{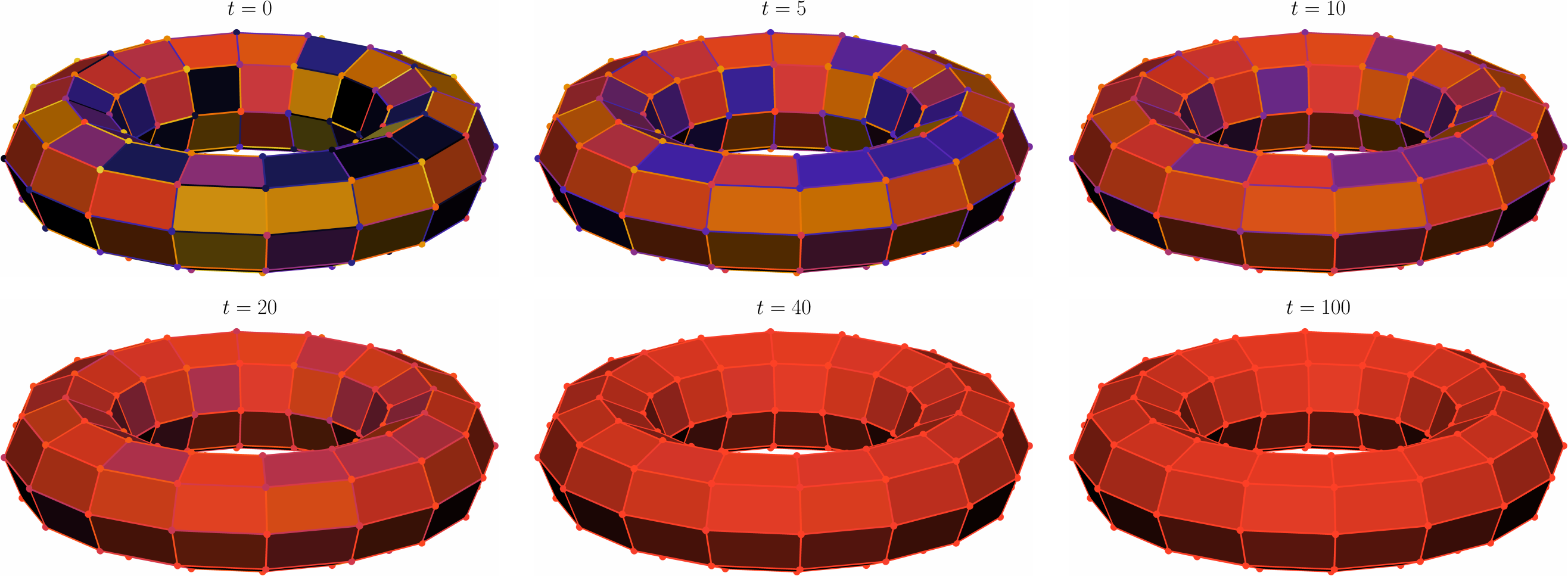}
\caption{Topological consensus on the homogeneous SLTT cell 
complex ($K=2$, $N_0 = 128$, $N_1 = 256$, $N_2 = 128$) under 
the DOSC dynamics with $\gamma_0 = \gamma_1 = 0.1$. Snapshots at 
$t = 0, 5, 10, 20, 40, 100$ show the simultaneous convergence 
of node, edge, and face signals to $\Phi^{*} = 0$.}
\label{fig:torus_homogeneous}
\end{figure*}

\begin{figure}[h]
  \includegraphics[width=\linewidth]{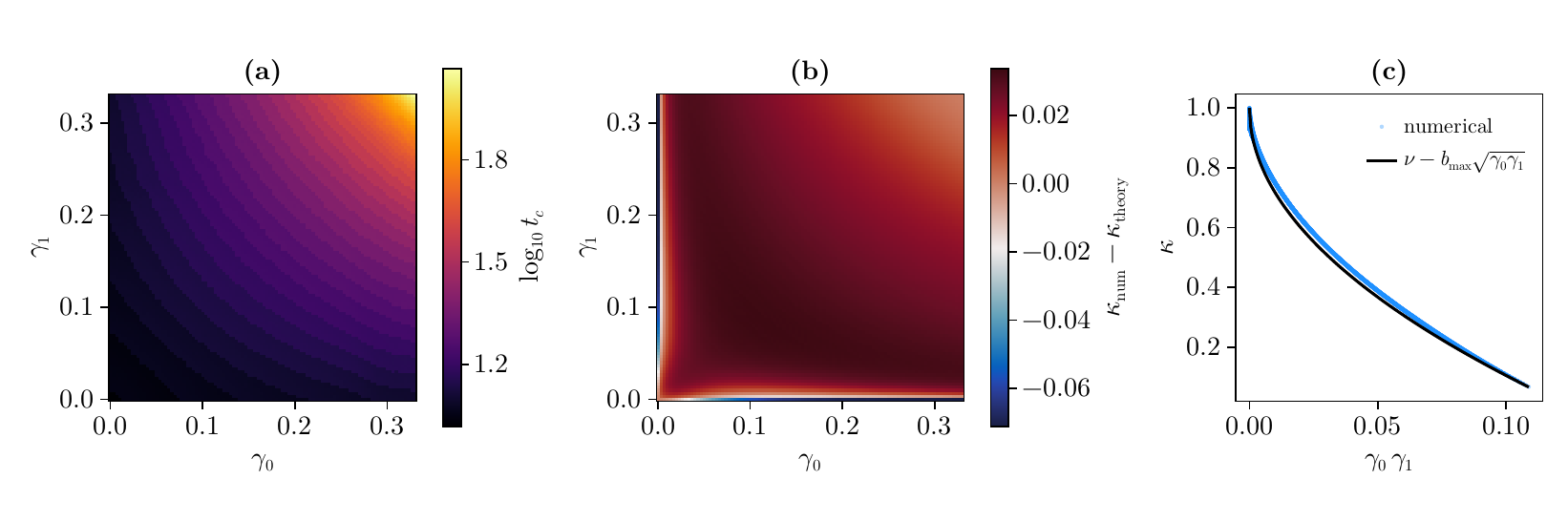}
  \caption{Convergence of the DOSC dynamics on the homogeneus SLTT as a
  function of the directional couplings $\gamma_0$ and $\gamma_1$. (a)
  Logarithm of the consensus time. (b) Difference between the numerically
  measured asymptotic decay rate and the theoretical prediction of
  ~\ref{thm:isotropic}. (c) Numerical decay rate versus the product $\gamma_0
  \gamma_1$, with the theoretical curve (solid line).}
  \label{torust11}
\end{figure}

Fig.~\ref{torust11} reports the previous same three diagnostics for the homogeneous SLTT cell complex ($K = 2$, $N_0 = 128$, $N_1 = 256$, $N_2 = 128$, $\nu = 1$), now with both couplings active across the node--edge and edge--face sectors. As for the Eulerian cycle, the consensus-time landscape in Fig.~\ref{torust11}(a) is symmetric about $\gamma_0 = \gamma_1$ and follows the hyperbolas of constant $\gamma_0\gamma_1$, and the rate collapses onto the curve $\nu - b_{\max }\sqrt{\gamma_0 \gamma_1}$ in Fig.~\ref{torust11}(c), now with $b_{\max } = 2 \sqrt{2}$. The agreement over the full stable range validates Theorem~\ref{thm:isotropic} in the genuinely two-dimensional setting, where the two sectors contribute distinct singular values yet share the same product law. The residual in Fig.~\ref{torust11}(b) reproduces the structure seen for the cycle (a positive interior bias from faster modes and a negative bias along the axes from the defective Dirac block) but is smaller in magnitude, confined to $[- 0.06, \text{ } 0.03]$. Both experiments thus confirm the product law and display the same finite-window artifacts, while the SLTT demonstrates that the prediction holds for a complex in which all three topological dimensions evolve and exchange information simultaneously.

%\ref{fig:torus_consensus_time} reports the consensus time landscape over the $(\gamma_0, \gamma_1)$ plane. The level curves of constant~$t_c$ are symmetric about the diagonal $\gamma_0 = \gamma_1$, and the fastest convergence is attained near the origin where the dynamics reduce to uncoupled local gradient descent. As either coupling parameter grows, the consensus time increases smoothly, and the contours follow the hyperbolic shape predicted by~\ref{thm:isotropic}: the slowest mode has rate $\nu - b_{\max} \sqrt{\gamma_0 \gamma_1}$, which depends on the two coupling parameters only through their product. The symmetry of the contours is the numerical signature of this product structure.

%\begin{figure}[t]
%\centering
%\includegraphics[width=0.5\textwidth]{consensus_time_torus1.pdf}
%\caption{Consensus time $t_c$ (logarithmic scale) as a function 
%of $(\gamma_0, \gamma_1)$ for the homogeneous SLTT. The level 
%curves are symmetric about the diagonal $\gamma_0 = \gamma_1$, 
%consistent with the product structure $\gamma_0 \gamma_1$ of 
%the isotropic stability condition~\ref{eq:stability_condition}.}
%\label{fig:torus_consensus_time}
%\end{figure}

\subsection{Heterogeneous multi-resource network}
\label{subsec:num_heterogeneous}

To explore the effect of heterogeneity on the convergence 
landscape, we keep the SLTT topology but assign distinct cost 
coefficients to each topological dimension, motivated by a 
schematic multi-resource infrastructure model:
\begin{equation}
\label{eq:hetero_obj}
F^{\sigma}_{0}(x) = \tfrac{1}{2} c^\sigma_0 x^{\top} x, 
\quad 
F^{e}_{1}(x) = \tfrac{1}{2} \ell^e_1 x^{\top} x, 
\quad 
F^{f}_{2}(x) = \tfrac{1}{2} r^f_2 x^{\top} x,
\end{equation}
where the coefficients $c^\sigma_0$, $\ell^e_1$, $r^f_2$ represent 
generation costs at nodes, transmission costs along edges, and 
loop-balancing penalties on faces, respectively. The cell complex 
naturally encodes Kirchhoff-type conservation laws through the 
exactness condition $B_1 B_2 = 0$. This extends the coordinated 
dispatch model of~\cite{rodriguezcamargo2023consensus} from a 
multiplex formulation to a topological formulation in which the 
admissibility condition replaces the consensus constraint and the 
boundary structure replaces external Lagrange multipliers.

Fig. \ref{fig:torus_heterogeneous} shows the evolution of all three 
families of signals on the heterogeneous SLTT. Despite the 
non-uniform cost coefficients breaking the symmetry across 
topological dimensions, the dynamics still converge to the 
synchronized optimum, in agreement with the global convergence 
result of Theorem~\ref{thm:global}, which holds for any positive 
$\nu$-strongly convex local objectives.

\begin{figure*}[t]
\centering
\includegraphics[width=0.95\textwidth]{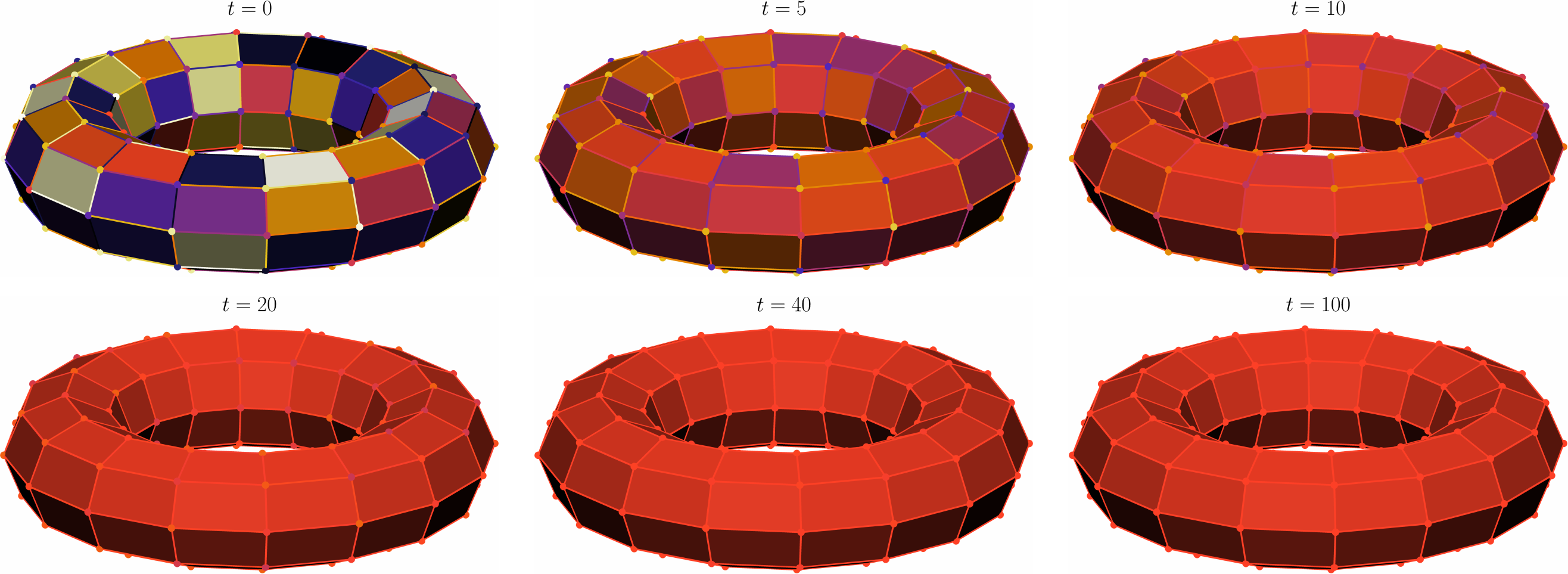}
\caption{Topological consensus on the heterogeneous SLTT 
(multi-resource model) under the DOSC dynamics. Snapshots at 
$t = 0, 5, 10, 20, 40, 100$ show the convergence of generation 
signals on nodes, flow signals on edges, and circulation signals 
on faces. The cross-dimensional Dirac coupling coordinates the 
three resource layers while the exactness condition 
$B_1 B_2 = 0$ enforces Kirchhoff-type conservation intrinsically.}
\label{fig:torus_heterogeneous}
\end{figure*}

Fig. \ref{fig:torus_hetero_consensus_time} shows the consensus time 
landscape over the $(\gamma_0, \gamma_1)$ plane for the 
heterogeneous case. The qualitative structure differs markedly 
from the homogeneous case: 
the contours are no longer symmetric about the diagonal but 
become predominantly vertical for moderate-to-large~$\gamma_0$, 
indicating that the node--edge coupling dominates convergence 
while the edge--face coupling plays a secondary role. The product 
structure $\gamma_0 \gamma_1$ of the isotropic stability 
condition is broken by the heterogeneous Hessians, and the 
spectral bottleneck shifts toward the~$B_1$ channel. The overall 
range of consensus times is also narrower than in the homogeneous 
case, suggesting that heterogeneity effectively regularizes the 
sensitivity to the coupling parameters near the stability 
boundary.

\begin{figure}[t]
\centering
\includegraphics[width=0.5\textwidth]{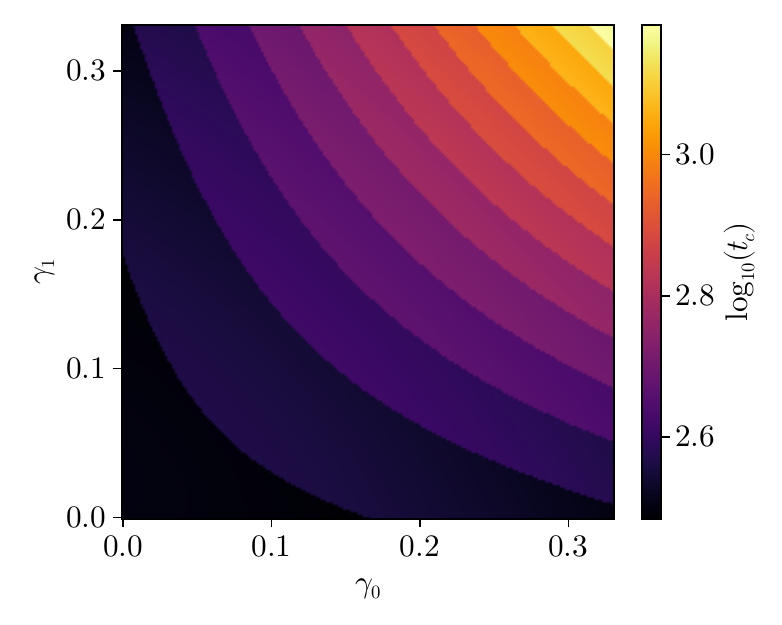}
\caption{Consensus time $t_c$ (logarithmic scale) as a function 
of $(\gamma_0, \gamma_1)$ for the heterogeneous multi-resource 
SLTT. The level curves are no longer symmetric about the 
diagonal $\gamma_0 = \gamma_1$, reflecting the breaking of the 
product structure of~\eqref{eq:stability_condition} by the 
heterogeneous Hessians.}
\label{fig:torus_hetero_consensus_time}
\end{figure}

%======================================================================
\section{Conclusions and Future Work}
\label{sec:summary}
%======================================================================

This paper has introduced a distributed optimization over simplicial complexes (DOSC) 
framework in which local gradient flows on each cell of a simplicial 
or cell complex are coupled through the topological Dirac operator. 
Three results structure the analysis: the synchronized equilibrium of 
the resulting dynamics coincides with the global minimizer of the 
aggregate strongly convex cost (\ref{thm:sync_opt}); local 
exponential stability is governed in the isotropic case by the 
explicit criterion 
$\gamma_0 \gamma_1 < \nu^2 / b_{\max}^2$ (\ref{thm:isotropic}); 
and global exponential convergence holds at the same rate under the same 
criterion through a weighted Lyapunov argument that symmetrizes the 
non-normal Dirac coupling (\ref{thm:global}). Numerical experiments on 
an Eulerian cycle, a homogeneous SLTT cell complex, and a 
heterogeneous multi-resource model on the same SLTT topology 
illustrate these results and reveal that heterogeneous Hessians 
break the product symmetry of the isotropic stability condition.

The DOSC framework completes a hierarchy of distributed optimization 
methods that progresses from single-layer graphs (graph 
Laplacian)~\cite{nedic2009distributed,nedic2018distributed}, through 
multiplex networks 
(supra-Laplacian)~\cite{rodriguezcamargo2023consensus,
RODRIGUEZCAMARGO20231217,11312412}, to general multilayer networks 
(tensor-Laplacian)~\cite{rodriguezcamargo2025}, and finally to 
higher-order topological structures (Dirac operator). The 
topological formulation is at once more general, in that it 
accommodates higher-order interactions absent from all previous 
levels, and structurally simpler, in that the algebraic topology of 
the complex absorbs constraints that earlier formulations imposed 
externally through Lagrange multipliers or dual variables. The graph 
Laplacian is recovered as the diagonal block~$L_0$ of $D^2$ when the 
complex reduces to a graph, connecting the framework to classical 
consensus.

Few questions remain open. The orientation equivariance required by 
cochain dynamics (\ref{ass:even}) is compatible with a nonzero optimum: 
the dimensional targets~$c_k^\sigma$ of~\eqref{eq:cost_form} drive the complex 
to a common optimal state $z^{*}$ with $\nabla h_k^\sigma(z^{*}) = c_k^\sigma$, which 
is nonzero whenever some $c_k^\sigma\neq 0$, so the framework performs genuine 
optimization rather than mere consensus to the origin. What equivariance 
does restrict is the \emph{granularity} of the targets: they must be 
constant across cells of a given dimension. Extending the framework to 
genuinely per-cell objectives, where the local minimizers differ from 
cell to cell and no exact synchronized equilibrium exists, is the most 
consequential open direction. The natural object there is the optimum restricted to the 
synchronization manifold, and recovering exact convergence to it appears 
to require augmenting the dynamics with a projection onto~$\mathcal{S}$ 
or, equivalently, a slow feedback that estimates the manifold-restricted 
gradient; quantifying the resulting consensus error and designing such 
projected dynamics is left to future work. Oscillatory reference states 
in the spirit of the original GTDS 
formulation~\cite{carletti2025gtds} offer a complementary route to 
time-varying targets. The topological admissibility Assumption \ref{ass:topology} 
restricts feasible unweighted topologies to a narrow class; relaxing 
it, possibly through a notion of partial topological synchronization 
in which only a subset of dimensions equilibrates, or through the 
weighted admissibility conditions 
of~\cite{carletti2025gtds}, would substantially enlarge the class of 
usable complexes. The MSF analysis of Section \ref{subsec:isotropic} does not 
extend cleanly to heterogeneous Hessians; deriving sharp stability 
conditions in this regime is a natural next step, suggested by the 
asymmetric numerical landscape observed 
in Section~\ref{subsec:num_heterogeneous}. Finally, combining 
multilayer and topological structure into a single 
tensor-Dirac formalism would enable distributed optimization on 
networks that are simultaneously layered and higher-order, with 
applications to topological data analysis~\cite{ISUFI2025109930, Lloyd2016, 9264674}, simplicial neural 
networks~\cite{ebli2020simplicialneuralnetworks}, topological signal processing on quantum computers~\cite{PhysRevApplied.23.054054}, and large-scale infrastructure systems with both lateral 
and dimensional couplings~\cite{PhysRevE.106.034319, Ma2025}.

%\section*{Acknowledgments}
%We would like to acknowledge the assistance of volunteers in putting
%together this example manuscript and supplement.

\section*{Statements and Declarations}

\subsection*{Conflict of Interest Statement}
The authors declare that they have no competing interests or conflict of interest.

\subsection*{Author Contributions Statement}
EMN, CRC, and AUR contributed equally to the conceptualization, methodology, theoretical analysis. CRC and AUR development the numerical experiments. All authors contributed to the writing and revision of the manuscript and approved the final version.

\subsection*{Funding Declaration}
Funding not applicable

% References are provided directly below.

\begin{appendices}

%======================================================================
\section{MSF derivation: reduction to the 2x2 block system}
\label{app:msf}
%======================================================================

We provide the detailed steps leading to the eigenvalue 
expression~\eqref{eq:eigenvalues} that underpins the local stability 
condition~\eqref{eq:stability_condition} of Theorem~\ref{thm:isotropic}. The 
key fact is that, by the invariance established 
in Lemma~\ref{lem:dirac_decomp} (iii), the spectrum of 
$\slashed{\mathcal{D}}$ is the disjoint union of its spectra on the 
node--edge sector $\operatorname{im}(D_{[1]})$, the edge--face sector 
$\operatorname{im}(D_{[2]})$, and the harmonic kernel $\ker(D)$; on each 
non-harmonic sector it further reduces to a family of independent 
$2\times 2$ blocks indexed by the singular values of $B_1$ and $B_2$. Let $B_k = U_k \Sigma_k V_k^{\top}$ denote the singular value 
decomposition of $B_k$ for $k = 1, 2$, with 
$U_k \in \mathbb{R}^{N_{k-1} \times N_{k-1}}$ and 
$V_k \in \mathbb{R}^{N_k \times N_k}$ orthogonal, and 
$\Sigma_k$ rectangular diagonal with non-zero singular values 
$\{ b_\alpha^{(k)} \}_\alpha$.

\emph{Node--edge sector.} For each non-zero singular value $b_\alpha$ 
of $B_1$, let $u_\alpha\in\mathbb{R}^{N_0}$ and 
$v_\alpha\in\mathbb{R}^{N_1}$ be the associated left and right singular 
vectors, so that $B_1 v_\alpha = b_\alpha u_\alpha$ and 
$B_1^{\top} u_\alpha = b_\alpha v_\alpha$. The two-dimensional subspace 
$\mathcal{S}_\alpha = \operatorname{span}\{(u_\alpha,0,0)^{\top}, 
(0,v_\alpha,0)^{\top}\}$ is contained in $\operatorname{im}(D_{[1]})$ 
and is invariant under $\slashed{\mathcal{D}}$: applying 
\eqref{eq:dirac_slashed},
\begin{equation*}
\slashed{\mathcal{D}}\,(u_\alpha,0,0)^{\top} 
= (0,\, \gamma_1 b_\alpha v_\alpha,\, 0)^{\top}, 
\qquad
\slashed{\mathcal{D}}\,(0,v_\alpha,0)^{\top} 
= (\gamma_0 b_\alpha u_\alpha,\, 0,\, 0)^{\top},
\end{equation*}
where the face block in the first identity vanishes because 
$B_2^{\top} v_\alpha = 0$ (as $v_\alpha\in\operatorname{im}(B_1^{\top})$ 
and $B_2^{\top}B_1^{\top}=0$). In the ordered basis 
$\{(u_\alpha,0,0)^{\top},(0,v_\alpha,0)^{\top}\}$ the restriction is 
therefore
\begin{equation}
\label{eq:block_app_ne}
\slashed{\mathcal{D}}\big|_{\mathcal{S}_\alpha} 
= \begin{bmatrix} 0 & \gamma_0 b_\alpha \\ \gamma_1 b_\alpha & 0 \end{bmatrix}.
\end{equation}
The two off-diagonal entries differ precisely because $\gamma_0$ weights 
the boundary map $B_1$ while $\gamma_1$ weights the coboundary map 
$B_1^{\top}$; this directional asymmetry is essential, as a 
sector-symmetric weighting (a single constant on both entries) would 
remove the product that appears below.

\emph{Edge--face sector.} For each non-zero singular value $b_\alpha$ of 
$B_2$, let $p_\alpha\in\mathbb{R}^{N_1}$ and 
$q_\alpha\in\mathbb{R}^{N_2}$ be the associated left and right singular 
vectors, so $B_2 q_\alpha = b_\alpha p_\alpha$ and 
$B_2^{\top} p_\alpha = b_\alpha q_\alpha$. By the same computation, with 
the node block of $\slashed{\mathcal{D}}\,(0,p_\alpha,0)^{\top}$ 
vanishing because $B_1 p_\alpha = 0$ (as 
$p_\alpha\in\operatorname{im}(B_2)$ and $B_1 B_2 = 0$), the subspace 
$\operatorname{span}\{(0,p_\alpha,0)^{\top},(0,0,q_\alpha)^{\top}\}$ is 
invariant and the restriction is
\begin{equation}
\label{eq:block_app_ef}
\begin{bmatrix} 0 & \gamma_0 b_\alpha \\ \gamma_1 b_\alpha & 0 \end{bmatrix},
\end{equation}
identical in form to~\eqref{eq:block_app_ne}. The node--edge blocks act 
on the gradient part of the edge space, $\operatorname{im}(B_1^{\top})$, 
and the edge--face blocks on the curl part, $\operatorname{im}(B_2)$; 
these are orthogonal by Lemma \ref{lem:dirac_decomp}, so the two families of 
blocks do not interact, and the reduction of the full 
$\mathcal{N}d$-dimensional problem to independent $2\times 2$ blocks is 
exact.

\emph{Combined block and eigenvalues.} Under the isotropic assumption 
$H_F = \nu I$, the stability matrix 
$M = -(H_F + \slashed{\mathcal{D}})$ leaves each two-dimensional 
subspace invariant, and on a block with singular value $b_\alpha$ its 
restriction reads
\begin{equation}
\label{eq:2x2_app}
\frac{d}{dt}\begin{bmatrix}\delta_1\\\delta_2\end{bmatrix}
= \begin{bmatrix} 
-\nu & -\gamma_0 b_\alpha \\ 
-\gamma_1 b_\alpha & -\nu 
\end{bmatrix}
\begin{bmatrix}\delta_1\\\delta_2\end{bmatrix},
\end{equation}
where $(\delta_1,\delta_2)$ are the coordinates of the perturbation 
along the two singular directions. The characteristic polynomial is
\begin{equation*}
(\lambda + \nu)^2 - \gamma_0\gamma_1 b_\alpha^2 = 0,
\end{equation*}
yielding the eigenvalues 
$\lambda_\pm^{(\alpha)} = -\nu \pm b_\alpha\sqrt{\gamma_0\gamma_1}$ 
of~\eqref{eq:eigenvalues}; both have strictly negative real part if and 
only if $b_\alpha\sqrt{\gamma_0\gamma_1} < \nu$. The binding constraint 
across all non-zero singular values of $B_1$ and $B_2$ is given by 
$b_{\max} = \max_{\alpha,k} b_\alpha^{(k)}$, yielding the stability 
criterion~\eqref{eq:stability_condition} and completing the proof 
of Theorem~\ref{thm:isotropic}. Finally, the diagonal similarity 
$S_\alpha = \diag(\sqrt{\gamma_0/\gamma_1},1)$ symmetrizes 
\eqref{eq:block_app_ne} and \eqref{eq:block_app_ef} into 
$\begin{psmallmatrix} 0 & b_\alpha\sqrt{\gamma_0\gamma_1} \\ 
b_\alpha\sqrt{\gamma_0\gamma_1} & 0 \end{psmallmatrix}$; assembling 
these block similarities yields the global weight $P$ 
of~\eqref{eq:weight_P} used in the proof of Theorem~\ref{thm:global}.

%======================================================================
\section{Topological admissibility: case analysis}
\label{app:existence}
%======================================================================

The topological admissibility condition $D \widehat{U} = 0$ of 
Assumption \ref{ass:topology} expands block-wise into the three coupled 
equations
\begin{align}
B_1^{\top} \widehat{u}_{1} &= 0, \label{eq:exist1} \\
B_1 \widehat{u}_{0} + B_2 \widehat{u}_{2} &= 0, \label{eq:exist2} \\
B_2^{\top} \widehat{u}_{1} &= 0, \label{eq:exist3}
\end{align}
where $\widehat{u}_{k} \in \mathbb{R}^{N_k}$ denotes the block of 
$\widehat{U}$ associated with $k$-cells, with 
$\widehat{u}^i_{0} = 1$ for all nodes and 
$\widehat{u}^\sigma_{k} \in \{-1, +1\}$ for all 
$k \geq 1$. We discuss the structural admissibility of three classes 
of complexes; the results of this appendix justify the topology 
choices made in Section~\ref{sec:numerical} and are due 
to~\cite{carletti2025gtds}.

\emph{Graphs ($K=1$).} With $N_2 = 0$, condition~\eqref{eq:exist3} is 
vacuous and~\eqref{eq:exist2} reduces to $B_1 \widehat{u}_{0} = 0$. 
Since $\widehat{u}^i_{0} = 1$ for every node, the entry of 
$B_1 \widehat{u}_{0}$ corresponding to edge $e_j = (i, k)$ is 
$\widehat{u}^i_{0} - \widehat{u}^k_{0} = 0$, so this condition 
is trivially satisfied. The non-trivial condition 
is~\eqref{eq:exist1}: for each node $i$,
\begin{equation*}
\sum_{e_j \,\text{incident to}\, i} [B_1^{\top}]_{j,i} \, \widehat{u}^j_{1} = 0.
\end{equation*}
With $\widehat{u}^j_{1} \in \{-1, +1\}$ and $[B_1^{\top}]_{j,i} = \pm 1$ 
depending on whether $i$ is the head or tail of edge $e_j$, this sum 
vanishes if and only if there exists an assignment of edge signs such 
that, at every node, the signed count of edges balances. An 
elementary combinatorial argument 
(see~\cite{carletti2025gtds}) shows that such an 
assignment exists if and only if every node has even degree, that 
is, if and only if the graph is Eulerian.

\emph{Square lattice tessellation of the torus ($K=2$).} For the 
SLTT, each face is bounded by exactly four edges. By choosing the 
canonical orientation in which each face inherits its boundary 
orientation consistently, the contribution of each face 
to~\eqref{eq:exist2} cancels pairwise across opposite edges. 
Condition~\eqref{eq:exist2} is then satisfied by 
$\widehat{u}_{2} \equiv 1$. Conditions~\eqref{eq:exist1} 
and~\eqref{eq:exist3} are satisfied because every node has degree 
four and every edge bounds exactly two faces, so the SLTT is 
admissible.

\emph{Unweighted triangulated surfaces ($K=2$).} For a triangulated 
surface where each face is a triangle bounded by three edges, the 
boundary contribution of each face to~\eqref{eq:exist2} involves 
three terms of magnitude $\pm 1$. At an edge incident to an odd 
number of triangles, the signed sum cannot vanish for any choice of 
signs $\widehat{u}^f_{2} \in \{-1, +1\}$. This obstruction 
prevents unweighted triangulated surfaces from satisfying the 
admissibility condition in general; see~\cite{carletti2025gtds} for 
a rigorous proof and for the weighted-simplicial relaxation 
discussed in Section~\ref{sec:summary}.

\end{appendices}
\end{document}